\documentclass[11pt,oneside]{amsart}
\usepackage[english]{babel}
\usepackage{enumitem}
\usepackage{csquotes}
\usepackage{bbm}
\usepackage{subcaption}
\usepackage{float}
\usepackage{fancyhdr}
\usepackage{etoolbox}
\usepackage[svgnames]{xcolor}
\usepackage[linkbordercolor=blue]{hyperref}
\usepackage{amsmath, amssymb}
\usepackage{mathrsfs}
\usepackage{palatino}
\usepackage{comment}
\usepackage{mathtools}


\newcommand{\mycomment}[1]{}

\usepackage{graphicx}
\usepackage{wrapfig}
\usepackage{caption}

 \usepackage[top=2.5cm, bottom=2.0cm,left=2.1cm,right=2.1cm]{geometry}

\usepackage[backend=biber,
            style=alphabetic,
            maxnames=4,
            firstinits=true ,
            doi=false,
            url=false,
            isbn=false,]{biblatex}
\usepackage{amsthm}

\theoremstyle{plain}
\newtheorem{theorem}{Theorem}[section]
\newtheorem{corollary}[theorem]{Corollary}
\newtheorem{assumption}[theorem]{Assumption}

\newtheorem{example}[theorem]{Example}
\newtheorem*{corollary*}{Corollary}
\newtheorem{lemma}[theorem]{Lemma}
\newtheorem{proposition}[theorem]{Proposition}

\newtheorem{remark}[theorem]{Remark}
\newtheorem{definition}[theorem]{Definition}

\numberwithin{equation}{section}

\definecolor{bordeaux}{HTML}{cc0000}

\newcommand{\whp}{{w.h.p.~}}
\newcommand{\wetahp}{{w.v.h.p.~}}
\newcommand{\wehp}{{w.v.h.p.~}}
\newcommand{\wethp}{{w.v.h.p.~}}
\newcommand{\iid}{{i.i.d.~}}
\newcommand{\rhs}{{r.h.s.~}}

\newcommand{\N}{\mathbb{N}}

\newcommand{\R}{\mathbb{R}}
\newcommand{\eps}{\varepsilon}
\newcommand{\E}{\mathbb{E}}
\newcommand{\var}{\mathbb{V}\hspace{-1.pt}{\rm ar}}
\newcommand{\cov}{\mathbb{C}\hspace{-0.pt}{\rm ov}}

\newcommand{\p}{\mathbb{P}}

\newcommand{\one}{\textup{\textbf{1}}}
\newcommand{\conv}{\xrightarrow[\textit{n} \to +\infty]{}}
\newcommand{\pconv}{\xrightarrow[\textit{n} \to +\infty]{\quad \p \quad}}
\newcommand{\dconv}{\xrightarrow[\textit{n} \to +\infty]{\quad {\rm d} \quad}}
\newcommand{\wconv}{\xrightarrow[\textit{n} \to +\infty]{\quad {w} \quad}}

\newcommand{\ber}{\text{Be}}

\newcommand{\matrixx}{{M_n}}

\newcommand{\matrixxprime}{{M'_n}}

\newcommand{\centeredentry}{{c}}
\newcommand{\centered}{C_n}
\newcommand{\centeredstar}{{C^*_n}}
\newcommand{\overlinecentered}{{\overline C_n}}
\newcommand{\tildecentered}{{\widetilde C_n}}

\newcommand{\deterministic}{S_n}
\newcommand{\tildedeterministic}{\widetilde{S}_n}

\newcommand{\term}{R}

\newcommand{\pat}{\mathfrak{p}}
\newcommand{\w}{\boldsymbol{w}}
\newcommand{\x}{{{v}}^+}
\newcommand{\y}{{{v}}^-}
\newcommand{\xx}{\boldsymbol{x}}
\newcommand{\yy}{\boldsymbol{y}}

\newcommand{\scaling}{s_n}

\newcommand{\op}{o_{\p}}
\newcommand{\Op}{O_{\p}}
\newcommand{\ohp}{o_{v.h.\p}}
\newcommand{\Ohp}{O_{v.h.\p}}
\DeclareMathOperator{\trace}{Tr}

\newcommand{\esd}[1]{\mu_{#1}}
\newcommand{\thetalimit}{\vartheta_z}
\usepackage{fvextra}
\usepackage[most]{tcolorbox}
\definecolor{ReviewOrange}{RGB}{190,90,0}
\definecolor{ReviewGreen}{RGB}{0,100,55}

\begin{document}
\title{Spectrum of Directed Inhomogeneous Random Graphs}

\author{Rajat Subhra Hazra}
\address{Mathematical Institute, Leiden University, Gorlaeus Gebouw,
BW-vleugel, Einsteinweg 55, 2333 CC Leiden, The Netherlands.}
\email{r.s.hazra@math.leidenuniv.nl}

\author{Giacomo Passuello}
\address{Delft Institute of Applied Mathematics (DIAM),
	Delft University of Technology, 
    Mekelweg 4, 2628 CD Delft, The Netherlands.}
\email{g.passuello@tudelft.nl}

% \date{\today}

\maketitle

	\begin{abstract}
We study the spectrum of the adjacency matrix $A_n$ of directed inhomogeneous random graphs on $n$ vertices. We assume that $A_n$ has independent entries and diverging average degree scale $s_n$. This framework includes, as special cases, the directed Chung--Lu random graph and directed stochastic block models. 
Assuming boundedness of the variance profile and that $s_n$ diverges faster than $\log(n)^2$, we show that the rank-one Chung--Lu model satisfies a non-homogeneous version of the circular law, which in some situations allows for an explicit expression. Moreover, under mild conditions, we identify the asymptotic singular value distribution using tools from free probability.
Finally, for finite-rank directed models, we establish that eigenvalues outside the bulk have joint Gaussian fluctuations at the scale $\sqrt{s_n/n}$, with an explicit covariance matrix. These results extend the theory of spectral outliers and their fluctuations to directed inhomogeneous random graphs.

   \par\bigskip\noindent
   {\it MSC 2020:} primary:  60B20, 05C80.
   \par\smallskip\noindent
   {\it Keywords:} random matrices, non-Hermitian matrices, inhomogeneous random graphs, circular law.\\

    \par\smallskip\noindent
    {\it Acknowledgments.} 
    GP is a member of GNAMPA-INdAM, acknowledges financial support through the GNAMPA projects %CUP E53C23001670001
``Ferromagnetism versus synchronization: how does disorder destroy universality?'' and 
``Redistribution models on networks’’, and thanks the Leiden Mathematical Institute for the kind hospitality. \\
\par\smallskip\noindent
    {\textbf{AI Statement}.} The first version of this paper was uploaded to arXiv in July 2026 and was written without the use of AI.
At a later stage, \texttt{ChatGPT 6.0 Astra} and \texttt{Claude Sonnet 5} were used to correct typos and improve the presentation.
\texttt{ChatGPT 6.0 Astra} suggested a simplified proof of Lemma \ref{lem:lem-W}; both approaches have been retained in the manuscript.
These tools also provided an additional verification of the proof of Corollary \ref{cor}, including a review of the cited proofs in \cite{BasakRudelson-invertibility,BasakRudelson19} with regard to the extension to inhomogeneous parameters.
The authors take full responsibility for the correctness of all statements and proofs.
\end{abstract}

  % \tableofcontents

\section{Introduction}

 Understanding the spectral properties of large random matrices has been a central challenge in probability theory and mathematical physics for decades. 
 In the Hermitian setting, Wigner demonstrated that the eigenvalues of a large symmetric matrix with independent entries, in the limit of large matrix size, follow the semicircular law. 
 Subsequent works have sharpened this picture, establishing local laws at optimal scales and local spectral universality; see, for example, \cite{erdos2017dynamical}. 
Random graphs provide adjacency and Laplacian matrices, for which this theory naturally finds application. 
In the undirected setting, where matrices are self-adjoint, it has been shown that after centering and rescaling, the empirical spectral distribution of the Erd\H{o}s--R\'enyi graph converges to the semicircle law. 
See \cite{EKetal, EKetalII} for other results on relevant spectral statistics. 
In particular, the mean adjacency matrix 
% in the Erd\H{o}s--R\'enyi case 
is rank-one, and this produces a leading eigenvalue separated from the bulk.

The separation between a random bulk and a deterministic low-rank structure has motivated a large body of work on outlier eigenvalues and finite-rank perturbations of random matrices; see, among others, \cite{BN, CDF, PRS}. 
% In random graph models, the same mechanism appears naturally through the expectation of the adjacency matrix. 
In the inhomogeneous Erd\H{o}s--R\'enyi setting, where this mechanism naturally appears through the expectation of the adjacency matrix, vertex-dependent connection probabilities lead to non-trivial limiting spectral distributions and 
% to outliers determined by the low-rank structure of the mean matrix. 
 general low-rank structures for the expectation matrix.
For example, \cite{CHHS, zhu2020graphon} identified limiting spectral distributions for adjacency and Laplacian matrices of inhomogeneous Erd\H{o}s--R\'enyi graphs, while \cite{CCH} studied eigenvalues outside the bulk in such models. The order of the largest eigenvalue  
of the adjacency matrix and its centering 
in both the sparse and dense regime of inhomogeneous Erd\H{o}s--R\'enyi graphs was studied in \cite{BBK1, BBK2} extending the classical work of \cite{FK}. Various extensions of the bulk results to kernel-based inhomogeneous random graphs were derived recently in \cite{cipriani2025spectrum}. These results show that vertex heterogeneity changes both the shape of the bulk and the behavior of the leading eigenvalues. 

In the directed case, the adjacency matrix is no longer self-adjoint, and its eigenvalues typically spread over a two-dimensional region in the complex plane. The basic universal object is then the circular law: for matrices with \iid centered entries with variance $1/n$, the empirical spectral distribution converges to the uniform distribution on the unit disk. This was conjectured by Girko and proved in great generality in \cite{TV, TVK, GT10}; see also the survey \cite{BC}. Further developments include sparse non-Hermitian matrices \cite{BasakRudelson19, RudelsonTikhomirov19, Revisited}, the single ring theorem \cite{GKZ}, and products or sums of non-Hermitian matrices \cite{GKT, KT}. In the very sparse regime, where the average degree is bounded, the limiting picture changes and an atom at the origin may appear \cite{SahSahasrabudheSawhney25}.

For directed random graphs, several related models have been studied. The circular law has been proved
for sparse directed Erd\H{o}s--R\'enyi graphs and directed regular graphs in \cite{BasakRudelson19-arxiv,Cook,sparsedigraph}. Eigenvalues outside the bulk have been analyzed for directed regular graphs in
\cite{B,Coste_Lambert_Zhou} and for the directed configuration model in \cite{Coste}. 
Spectral edge and eigenvector localization for directed Erd\H{o}s--R\'enyi graphs have been investigated in \cite{BBK1,CritErd}. In
parallel, non-Hermitian matrices with variance profiles have been studied through deterministic
equivalents, master equations, local laws, and free-probabilistic methods
\cite{CookI,COOKII,AltErdosKruger,XiYangYin-local,BigotMale,HachemLouvaris}. We also mention the recent works \cite{Dumitriu_2024,DumitriuWangWangZhu25} on extreme singular values of sparse rectangular matrices, including inhomogeneous models and the critical regime in which singular-value outliers emerge. These works provide
important tools for treating non-identically distributed entries, but the spectral theory of directed
random graphs with degree heterogeneity and finite-rank mean structure remains less developed.

In this paper we study the spectrum of adjacency matrices of directed inhomogeneous random graphs. Our main example is the directed Chung--Lu model. Each vertex $x\in [n]$ is assigned two positive weights $w_x^{+}$ and $w_x^{-}$ representing the outgoing and incoming tendency, and the directed edge $x\to y$ is present independently with probability proportional to 
$$\scaling \frac{w_x^{+}w_y^{-}}{\w},\qquad \w=\sum_{x=1}^n w_x^{+}=\sum_{x=1}^n w_x^{-}.$$
The parameter $\scaling$ controls the average degree and is allowed to diverge much more slowly than $n$. We also
consider finite-rank directed models in which the expectation of the adjacency matrix decomposes into a
finite sum of rank-one components. This includes directed stochastic block models as a basic example. Models of this form provide a natural random environment for stochastic dynamics. Recently, it has been shown that the simple random walk on such graphs exhibits cutoff with high probability \cite{BP,BPQ}. In reversible settings, the mixing behavior of random walks is closely related to the spectral properties of the graph. 

Our contributions are threefold. First, for the rank-one Chung--Lu model, we identify the limiting empirical spectral distribution of $A_n / \sqrt{s_n}$. The limit is an inhomogeneous circular-law type distribution and can be described through a weighted Ginibre ensemble. In particular, the radial distribution can be expressed in terms of the limiting distribution of the product of the in- and out-weights, using the $S$-transform (in some cases). We also identify the limiting singular value distribution under a linear transformation of $A_n / \sqrt{s_n}$ and under an additional matching condition on in- and out-weight profiles.

Second, we characterize spectral outliers. In the rank-one Chung--Lu model, the leading eigenvalue is located at scale $s_n$, close to the unique nonzero eigenvalue of $\mathbb{E}\left[A_n\right]$, while all remaining eigenvalues are confined to scale $\sqrt{s_n}$ with very high probability. We also obtain a corresponding estimate for the transition matrix of the simple random walk on the directed graph: apart from the trivial eigenvalue at $1$, the spectrum is contained in a disk of radius $O(1/\sqrt{\scaling})$ with very high probability.

Third, we establish Gaussian fluctuations for the outlier eigenvalues. In the rank-one Chung--Lu model, the leading eigenvalue satisfies a central limit theorem after centering and scaling. We then extend this result to finite-rank directed models: the $r$ outliers generated by the rank-$r$ expectation matrix converge jointly to a centered Gaussian vector, with an explicit covariance matrix determined by the limiting empirical profile of the left and right weight vectors.

Several modifications are needed compared with the self-adjoint and homogeneous settings. For the bulk, we use Girko's Hermitization method, reducing convergence of the empirical spectral distribution to the analysis of singular values of shifted matrices. In the sparse and inhomogeneous setting, this requires lower bounds on the least singular value, control of intermediate singular values, and a comparison with a Gaussian matrix with the corresponding variance profile. The limiting distribution is then identified using asymptotic freeness and $R$-diagonal operators. For the outliers, self-adjoint perturbation tools such as Weyl interlacing or Hoffman--Wielandt inequalities are no longer available. We instead use Bauer--Fike perturbation theory, together with high-trace estimates for the centered non-Hermitian matrix. Finally, the fluctuation results rely on an eigenvalue fixed-point expansion: after suitable high-probability reductions, the leading random contribution is a sum of independent terms, to which a Lindeberg central limit theorem can be applied.\\

\paragraph*{\bf Outline} 
Section~\ref{sec:results} introduces the models and the assumptions, and states the main results, together with examples and a discussion of the methods and open points.
Section~\ref{sec:bulk} establishes the bulk limit in the polynomial degree regime and proves the singular value limits.
Section~\ref{sec:sparse} extends the bulk analysis to the sparser regime for $\scaling$.
Section~\ref{sec:sp-outlier-1} is devoted to the characterization of the outlier in the rank-one case and its Gaussian fluctuations.
Finally, Section~\ref{sec:sp-outliers-r} extends the outlier analysis to the finite-rank setting and establishes joint Gaussian fluctuations.\\

\paragraph*{\bf Notation} Before proceeding with the models and the results, we fix some notation.
For an $n \times n$ complex-valued matrix $\matrixx$, we denote by $(\lambda_i(\matrixx))_{1 \le i \le n}$ the sequence of (complex) eigenvalues of $\matrixx$, ordered so that 
$|\lambda_1(\matrixx)|\ge |\lambda_2(\matrixx)|\ge \cdots \ge |\lambda_n(\matrixx)|\ge 0,$
and by $(\sigma_i(\matrixx))_{1 \le i \le n}$ its (decreasingly) ordered singular values, i.e.,  for any $i$, $\sigma_i(\matrixx)$ is the square root of the $i$-th largest eigenvalue of $M_nM_n^*$.
Its spectral norm is $\|\matrixx\|: = \sigma_{1}(\matrixx)$. 
If $\matrixx$ is self-adjoint, $\|\matrixx\|=|\lambda_{1}(\matrixx)|$.
The empirical spectral distribution (ESD ) and the singular value distribution of an $n \times n$ matrix $M_n$ are defined as 
\begin{equation*}
    \mu_{M_n}= \frac{1}{n}\sum_{i =1}^n \delta_{\lambda_i(M_n)} \quad \text{ and } \quad     \nu_{M_n}= \frac{1}{n}\sum_{i =1}^n \delta_{\sigma_i(M_n)}.
\end{equation*}
We will be interested in their weak limit in probability. We write $\mu_n \wconv \mu$ to say that a sequence $(\mu_n)_{n \in \mathbb{N}}$ of probability distributions converges weakly to some limit $\mu$, and we use the standard Landau notations $\gg , \ll, O(\cdot), o(\cdot),\sim,\asymp,\lesssim$. In addition, for two non-negative random variables $X_n$ and $Y_n$, we will write $X_n=\Ohp(Y_n)$ (the subscript stands for \textit{very high} probability) if there exist $K>0$ and $\eta>1$ such that, 
\begin{equation*}
    \p(X_n \ge K\ Y_n) \le e^{-(\log(n))^\eta},
\end{equation*}
and we will write $X_n=\ohp(Y_n)$ if, for every $\delta>0$, there exists $\eta>1$ such that
\begin{equation*}
    \p(X_n \ge \delta\ Y_n) \le e^{-(\log(n))^\eta}.
\end{equation*}
Finally, we will say that an event holds with very high probability (w.v.h.p.) if there exists $\eta>1$ such that the probability of its complement decays as in the above display. This strengthens the notation $\op(\cdot)$, which means that $|X_n/Y_n|$ vanishes in probability. We also say that $X_n=\Op(Y_n)$ if $|X_n/Y_n|$ is tight. As a last remark, in the statements many absolute constants will appear, none of them having an important role in the arguments.

\section{Models and Results}\label{sec:results}
In this section we introduce the directed inhomogeneous random graph models considered in the paper and state our main results. We begin with the rank-one Chung--Lu model, for which we prove a bulk limit, a singular value limit, the existence of an outlier, a spectral estimate for the random walk transition matrix, and Gaussian fluctuations of the leading eigenvalue. We then state the corresponding outlier and fluctuation results for a finite-rank directed model.

\subsection{Directed Chung--Lu model}

For every $n\in \N$, let
$(w_x^+,w_x^-)_{x\in[n]}$ be non-negative bi-weights satisfying the balance condition
\[
        \sum_{x=1}^n w_x^+=\sum_{x=1}^n w_x^-=:\w.
\]
Conditionally on the weights, the entries of the adjacency matrix $A_n=(a_{xy})_{x,y\in[n]}$ are independent Bernoulli random variables with
\begin{equation}
\label{eq:sp-chung-lu}
        p_{xy}:=\mathbb P(a_{xy}=1)=\frac{s_n w_x^+w_y^-}{\w},
        \qquad x,y\in[n].
\end{equation} 
In our setting for large $n$ we will have $p_{x,y}<1$, so that no truncation is needed.
We work under the following assumption on the joint empirical distribution of the weights. An assumption on $\scaling$ will be given later.

\begin{assumption}[Rank-one weight profile]
\label{ass:rank-one}
There exist constants $0<c<C<\infty$ such that
\begin{equation}
\label{eq:weight-bounds}
        c\le w_x^\pm\le C,\qquad x\in[n],\ n\ge1.
\end{equation}
Moreover there exists a compactly supported probability measure
$\rho=\rho_{w^+,w^-}$ on $(0,\infty)^2$ such that
\begin{equation}
\label{eq:rho-one}
        \frac1n\sum_{x=1}^n
        \delta_{\sqrt{n/\w_n}(w_x^+,w_x^-)}
        \wconv \rho .
\end{equation}
\end{assumption}
For $\#\in\{+,-\}$, define the diagonal matrices
\begin{equation}
\bar{W}_n^{\#} =
{\rm Diag}
\left(\sqrt{\frac{n}{\w}}\,w_1^{\#},\dots,
\sqrt{\frac{n}{\w}}\,w_n^{\#}\right),
\end{equation}
and set
\begin{equation}
\bar W_n=\sqrt{\bar W_n^+\bar W_n^-}.
\end{equation}
By Assumption \ref{ass:rank-one}, the empirical distribution of the diagonal entries of $\bar W_n$ converges weakly to a compactly supported probability measure $\bar\rho$. Equivalently, if $(X^+,X^-)$ has law $\rho$, then $\bar\rho$ is the law of $\sqrt{X^+X^-}$.

\subsubsection{Bulk}
For the bulk analysis, we work under one of the two following assumptions, which both imply that average degrees diverge sufficiently fast.

\begin{assumption}\label{assumption:bulk-poly} There exists $\alpha \in (0,1)$ such that $\scaling \sim n^\alpha$.
\end{assumption}

\begin{assumption}
\label{assumption:bulk-sparse}
    $\log^2(n) \ll \scaling \ll n$, and, 
    for any $x \le \lfloor \tfrac n 2 \rfloor$, $w_{x}^+=w_{x+\lfloor n/2\rfloor }^{+}$ and $w_{x}^-=w_{x+\lfloor n/2\rfloor }^{-}$.
    % and $w_{x}^-=w_{x+\lfloor n/2\rfloor }^{-}$ 
\end{assumption}
\begin{remark}
Under the assumption $\log^2(n) \ll \scaling$ , the graph is \whp strongly connected. Indeed, the threshold for this property is attained when the connection probability is uniformly of order at least $\log(n)/n$ (see \cite{CF}). The additional condition on the weights in Assumption \ref{assumption:bulk-sparse} is technical and is used only in the proof of the least singular value estimate in the sparser regime in Section \ref{sec:sparse}. We expect the conclusion to remain valid without it.
\end{remark}

For $z\in\mathbb C$, let $\vartheta_z$ denote the deterministic weak limit, as $n\to\infty$, of the empirical spectral distribution of the Hermitized Gaussian matrix
\begin{equation}
\label{eq:her}
H_n^g(z)=
\begin{pmatrix}
0 &
\frac{1}{\sqrt n}(\bar W_n^+)^{1/2}G_n(\bar W_n^-)^{1/2}-zI_n\\
\frac{1}{\sqrt n}(\bar W_n^-)^{1/2}G_n^*(\bar W_n^+)^{1/2}-\bar zI_n
& 0
\end{pmatrix},
\end{equation}
where $G_n$ is a Ginibre matrix with \iid complex centered Gaussian entries of variance $1$.

\begin{theorem}[Bulk limit]
\label{thm:bulk}
Suppose that Assumption \ref{ass:rank-one} holds, and that either Assumption \ref{assumption:bulk-poly} or Assumption \ref{assumption:bulk-sparse} holds. Then the empirical spectral distribution of $A_n/\sqrt{\scaling}$ converges weakly in probability to a deterministic probability measure $\mu$ on $\mathbb C$. The measure $\mu$ is characterized by its logarithmic potential as follows:
\begin{equation}
\int_{\mathbb C}\log|\zeta-z|\,d\mu(\zeta)
=
\int_{\mathbb R}\log|x|\,d\vartheta_z(x),
\qquad z\in\mathbb C.
\end{equation}
Moreover, $\mu$ has compact support.
\end{theorem}

\noindent
In order to state the next result, we define the Stieltjes transform and the $S$-transform of a probability distribution $\tau$. The former is defined by
\begin{equation}
    \mathcal{G}_\tau(z) = \int_{\R} \frac{1}{z-x} \, d\tau(x), \qquad z \in \mathbb{C}^+.
\end{equation}
 The $S$-transform was initially introduced in \cite{Voiculescu} and \cite{BerVoi} for probability distributions with non-zero mean. Later \cite{RaoSpeichert} extended the definition to distributions having zero mean and all moments and \cite{ArizmendiPerezAbreu} to the unbounded support case.
We first introduce
\begin{equation}
    \psi_\tau(w) = \frac{1}{w} G_\tau\left(\frac 1 w\right) -1 = \int_\R \frac{wx}{1-wx} \, d\tau(x), \qquad w\in \mathbb{C}^-,
\end{equation}
and later define the $S$-transform of $\tau$, for $y$ in a small complex disk centered at the origin, as
\begin{equation} \label{eq:S}
    \mathcal{S}_\tau(y) = \psi_\tau^{-1}(y)\frac{1+y}{y}.
\end{equation}
This definition was later extended to distributions supported on $\mathbb{R}$ in \cite{ArizmendiPerezAbreu}.
In the statements, for a given measure $\nu$, we denote by $\nu^2$ and $\sqrt{\nu}$ the push-forward of $\nu$ via the maps $x \mapsto x^2$ and $x\mapsto \sqrt{x}$ respectively.

\begin{corollary}[Radial description of the bulk]
\label{cor:bulk}
Under the assumptions of Theorem \ref{thm:bulk}, the limiting empirical spectral distribution of $A_n/\sqrt{\scaling}$ is the same as that of
\begin{equation}
\frac1{\sqrt n}G_n\bar W_n .
\end{equation}
Moreover, the limiting measure $\mu$ is radial and satisfies
\begin{equation}
\mu\left(\mathcal B(0,F(t))\right)=t,
\qquad 0<t\le 1,
\end{equation}
where
\begin{equation}
F(t)=\sqrt{\frac{t}{\mathcal S_{\bar\rho^2}(t-1)}} .
\end{equation}
Consequently, whenever $F$ is differentiable and strictly increasing, $\mu$ has density
\begin{equation}
d\mu(z)=
\begin{cases}
\frac{1}{2\pi |z| F'(F^{-1}(|z|))}\,d z,
& |z|\in(0,F(1)],\\
0, & \text{otherwise}.
\end{cases}
\end{equation}
\end{corollary}

Our next statement provides a characterization of the singular value distribution for $A_n$ and a similar transformation of $A_n$.
\begin{theorem} \label{thm:SV}
Suppose that the assumptions in Theorem \ref{thm:bulk} hold. Let
\begin{equation}
\label{eq:B}
\bar P_n =
{\rm Diag}\left(
\sqrt[4]{\frac{w_1^-}{w_1^+}},\dots,
\sqrt[4]{\frac{w_n^-}{w_n^+}}
\right),
\qquad
B_n=\bar P_n A_n\bar P_n^{-1}.
\end{equation}
Then
\begin{equation}
\nu_{B_n/\sqrt{\scaling}}
\wconv
\sqrt{\bar\rho\boxtimes m\boxtimes\bar\rho}
\end{equation}
in probability, where $\boxtimes$ denotes free multiplicative convolution and $m$ is the Marchenko--Pastur distribution with parameter one, given by
\begin{equation} \label{eq:marchenko-1}
    dm(x)=\frac1{2\pi} \sqrt\frac{{4-x}}{x} \one_{(0,4)}(x)dx.
\end{equation}
If, in addition, the empirical distributions of $\bar W_n^+$ and $\bar W_n^-$ converge weakly to the same limit $\tilde\rho$, then
\begin{equation}
\nu_{A_n/\sqrt{\scaling}}
\wconv
\sqrt{\tilde\rho\boxtimes m\boxtimes\tilde\rho}
\end{equation}
in probability.
\end{theorem}

\subsubsection{Outlier}

For the outlier and fluctuation results we require the following condition on $\scaling$.
\begin{assumption}\label{assumption:xi} 
There exists $\xi>8$ such that $ \log^\xi(n)\ll \scaling \ll n$.
\end{assumption}
\begin{theorem}[Existence of outlier - rank-one case] \label{thm:LLN-A}
Consider the Chung--Lu model with $p_{xy}$ as in Eq.\ \eqref{eq:sp-chung-lu}.
If Assumptions \ref{ass:rank-one} and \ref{assumption:xi} hold, then there exists $K>0$ and $\eta>1$ such that
\begin{equation}
    \p\left(\max\left\{\left|\lambda_1(A_n)- \lambda_1(\E[A_n]) \right|,\max_{ 2 \le j\le n}|\lambda_j(A_n)|\right\}\ge K\sqrt{\scaling} \right)\le e^{-(\log(n))^{\eta}}.
\end{equation}
\end{theorem}
\begin{remark}The conclusion of Theorem \ref{thm:LLN-A} can equivalently be written as
\begin{equation}
\label{eq:LLN-A-Ohp}
\max\left\{
|\lambda_1(A_n)-\lambda_1(\E[A_n])|,
\max_{2\le j\le n}|\lambda_j(A_n)|
\right\}
=
\Ohp(\sqrt{\scaling}).
\end{equation}
    Since $\E[A_n]$ is rank-one with entries $\scaling \tfrac{w_x^+w_y^-}{\w}$, we can compute $$\lambda_1(\E[A_n])=\sum_{x\in [n]} \scaling \tfrac{w_x^+w_x^-}{\w}.$$
 Thus, Theorem \ref{thm:LLN-A} shows that the leading eigenvalue lives at scale $\scaling$, whereas all other eigenvalues live at scale $\sqrt{\scaling}$ with very high probability.
\end{remark}
\begin{remark}
The assumption $\log^\xi n\ll \scaling$, with $\xi>8$, is used to obtain estimates with very high probability. This form is needed later in the fluctuation analysis, (see Lemma \ref{lem:1}) where the error bounds have to hold on events whose complements are smaller than any polynomial power of $n$. 
\end{remark}

A similar theorem can be stated for the transition matrix of the simple random walk. As observed in Lemma \ref{lem:concentration}, degrees are uniformly positive with high probability, so that the diagonal matrix $D_n$ with entries $D_x^+$, for $x \in [n]$, where $D_x^+$ is the out-degree of $x$, is with high probability invertible. If this is not the case, we can set 
\begin{equation}
    D^{-1}_n(x,x) = \begin{cases} 1/D_x^+ & \text{ if $D_x^+>0$}, \\ 0 &\text{ otherwise}.
    \end{cases}
\end{equation}
\begin{theorem}[Existence of outlier, random walk] 
Consider the Chung--Lu model with $p_{xy}$ as in Eq.\ \eqref{eq:sp-chung-lu}.
Let $T_n=D_n^{-1}A_n$.  If Assumptions \ref{ass:rank-one} and \ref{assumption:xi} hold, then
\label{thm:LLN-T}
\begin{equation}
  \max_{ 2 \le j\le n}|\lambda_j(T_n)| = \Ohp\left(\frac{1}{\sqrt{\scaling}}\right).
\end{equation}
\end{theorem}

\begin{theorem}[Gaussian fluctuations, rank-one case] \label{thm:CLT-A}
Consider the Chung--Lu model with $p_{xy}$ as in Eq.\ \eqref{eq:sp-chung-lu}.
If Assumptions \ref{ass:rank-one} and \ref{assumption:xi} hold,
then
\begin{equation}
    \sqrt{\frac{n}{\scaling}}\left(\lambda_1(A_n)- \E[\lambda_1(A_n)] \right)\dconv G,
   \end{equation}
where $G$ is a  centered Gaussian random variable with variance 
    \begin{equation} \label{eq:sigma-integral}
        \sigma^2=
        \frac{\displaystyle  \left(\int_{\R^{2}_+} (x^+)^2x^- \,d\rho(x^+,x^-)\right)\left(\int_{\R^{2}_+} x^+(x^-)^2 \,d\rho(x^+,x^-)\right)}{\displaystyle  \left(\int_{\R^{2}_+} x^+x^- \,d\rho(x^+,x^-)\right)^2}.
    \end{equation}
\end{theorem}
Notice that the variance in \eqref{eq:sigma-integral} is finite, due to the compact support of $\rho$.

\subsection{Higher rank model}

Our second model is more general.
Let $v^\pm_1,\dots,v^\pm_r \in \R^n$ be a family of $2r$ bi-orthogonal vectors, i.e.,   such that $(v_i^+)^tv_j^- = \delta_{ij}$ for $i,j=1,\dots,r$. We can assume without loss of generality $\|v_i^+\| =1$ for each $i=1,\dots,r$.
Let $\theta_1>\theta_{2}>\dots>\theta_r$ be positive constants, and set
\begin{equation} \label{eq:rank-m}
    p_{x,y} = \scaling\sum_{j=1}^r \theta_j v^+_{j}(x) v^-_{j}(y), \qquad x,y \in [n],
\end{equation} where, for $x \in [n]$, $v_i^\pm(x)$ denotes the $x$-th entry of $v_i^\pm$,  and $\scaling$ satisfies Assumption \ref{assumption:xi}.
The rank-one model is given by $r=1$, $\theta_1=\sum_{x \in [n]} w_x^+ w_x^- /\w$, and $v^+_1$ to be the normalization of $(w_1^+,\dots,w_n^+)$ and $v_1^-={(\sum_{x \in [n]} (w_x^+)^2)^{\frac12}} (w_1^-,\dots,w_n^-)/(\sum_{x \in [n]} w_x^+ w_x^- )$. However, notice that $\theta_1$ here can depend on $n$.

The representation \eqref{eq:rank-m} applies to real
diagonalizable mean matrices whose nonzero eigenvalues
are distinct and positive, with the stated bounds
on their left and right eigenvectors.
A non-symmetric admissible example can be obtained
from a symmetric admissible finite-rank mean matrix
$B_n$ by a positive diagonal similarity
$D_nB_nD_n^{-1}$, where the diagonal entries of $D_n$
are uniformly bounded above and away from zero.
This preserves nonnegative entries, rank and
eigenvalues, and preserves delocalization of the
correspondingly rescaled eigenvectors. One must also
choose the diagonal profiles so that the joint
empirical limit in Assumption~\ref{ass:rank-m} exists.
For example, profiles constant on finitely many
vertex classes of convergent proportions have this
property. The probability upper bound holds for
large $n$ because $s_n/n\to0$.

Dealing with the connection probabilities of a random graph, some positivity and boundedness constraints have to be satisfied.
Moreover Assumption \ref{ass:rank-one} is generalized accordingly.
\begin{assumption} \label{ass:rank-m}
For each $1\le i\le r$, the entries of $\sqrt n v_i^\pm$ are uniformly bounded in $n$. Moreover, the quantities in \eqref{eq:rank-m} define probabilities.
Finally, assume there exists a compactly supported probability measure $\rho_r=\rho_{v_1^\pm,\dots,v_r^\pm}$ on $\mathbb \R^{2r}$ such that
\begin{equation}
\frac1n\sum_{x=1}^n
\delta_{
(\sqrt n v_1^+(x),\dots,\sqrt n v_r^+(x),
\sqrt n v_1^-(x),\dots,\sqrt n v_r^-(x))
}
\wconv \rho_r\,.
\end{equation}
% and on the support of $\rho_{r}$ the function $f:\R^{2r} \to \mathbb{R}$, $f(\boldsymbol{z}^+,\boldsymbol{z}^-)=  \sum_{k=1}^r \theta_k z_k^+ z_k^-$ has values in $[0,1]$.}
\end{assumption}
\subsubsection{Outliers}
\begin{theorem}[Existence of outliers, rank-$r$ case] \label{thm:LLN-A-rank-m}
Consider the model where $p_{xy}$ is as in Eq.\ \eqref{eq:rank-m}.
If Assumptions \ref{assumption:xi} and \ref{ass:rank-m} hold, then
\begin{equation}
    \max\left\{\max_{ 1 \le i\le r}\left|\lambda_i(A_n)- \lambda_i(\E[A_n]) \right|,\max_{ r+1 \le j\le n}|\lambda_j(A_n)|\right\}=\Ohp(\sqrt{\scaling}).
\end{equation}
\end{theorem}

\begin{theorem}[Gaussian fluctuations, rank-$r$ case] \label{thm:CLT-A-rank-m}
Consider the model where $p_{xy}$ is as in Eq.\ \eqref{eq:rank-m}.
Let $f:\R^{2r}\to \mathbb{R}$ be the function defined by $f(\boldsymbol{z}^+,\boldsymbol{z}^-)=  \sum_{k=1}^r \theta_k z_k^+ z_k^-$, for any $(\boldsymbol{z}^+,\boldsymbol{z}^-) \in \R^{2r}$.
If Assumptions \ref{assumption:xi} and \ref{ass:rank-m} hold, then
\begin{equation}
    \sqrt{\frac{n}{\scaling}}\left(\lambda_i(A_n)- \E[\lambda_i(A_n)] \right)_{1\le i \le r}\dconv (G^{(i)})_{1\le i \le r},
\end{equation}
where $(G^{(i)})_{1\le i \le r}$ is a centered Gaussian random vector with covariances given, for $i,j=1,\dots,r$, by
 \begin{equation} \label{eq:cov}
        \cov (G^{(i)},G^{(j)})=
        % \frac{
        \displaystyle 
        \int_{\R^{2r}}\int_{\R^{2r}}
        x^-_i x_j^-\,  f(\boldsymbol{x}^+,\boldsymbol{y}^-) \, y_i^+y^+_j
        \,d\rho_r(\boldsymbol{x}^+,\boldsymbol{x}^-)
        \,d\rho_r(\boldsymbol{ y}^+,\boldsymbol{y}^-).
    \end{equation}
\end{theorem}

\subsection{Examples} We provide some instances of weight distributions allowing for explicit computations.
\begin{example}[Girko  sombrero distribution] \label{ex:sombrero} 
An interesting example is given by weight distributions such that $\bar \rho = \frac{1}{2}(\delta_{\sqrt{a}}+\delta_{\sqrt{{b}}})$ where $0<a<b<+\infty$. 
In that case it is possible to show that
\begin{equation}
S_{\bar \rho^2}(t-1)
=
\frac{(a+b)(2(t-1)+1)
-\sqrt{(a-b)^2(2(t-1)+1)^2+4ab}}
     {4ab\,(t-1)},
     \end{equation}
and Corollary \ref{cor:bulk} provides a pipeline to derive the limiting distribution in this case, which is the Girko  sombrero distribution presented in \cite[Section 26.12]{Girko-canonical}, with density
\begin{equation}
      d\mu(z)= 
    \frac{1}{2\pi ab}\left((a+b)-\frac{|z|^2(a-b)^2}{\sqrt{|z|^4(a-b)^2 + a^2b^2}}\right) \one_{\left\{|z|\le \sqrt{\frac{a+b}{2}}\right\}}\; d^2z.
\end{equation}
A plot of the resulting spectrum is given in Figure \ref{fig:test}(A).
\begin{example}[If $\bar\rho^2$ is Marchenko--Pastur]
Let $0<\kappa<1$ and let $\bar\rho^2$ be the
Marchenko--Pastur law of mean one and parameter $\kappa$.
Its $S$-transform is
\[
\mathcal S_{\bar\rho^2}(t-1)=(1-\kappa+\kappa t)^{-1}.
\]
Consequently,
\[
F(t)=\sqrt{t(1-\kappa+\kappa t)},\qquad
F^{-1}(r)=\frac{\sqrt{(1-\kappa)^2+4\kappa r^2}
 -(1-\kappa)}{2\kappa},\quad 0\le r\le1.
\]
The limiting planar density is
\[
d\mu(z)=\frac{\one_{\{|z|\le1\}}}
{\pi\sqrt{(1-\kappa)^2+4\kappa|z|^2}}\,d^2z.
\]
For $\kappa=0$, the input law is $\delta_1$ and the
limit is uniform on the unit disk. The formal endpoint
$\kappa=1$ gives density $1/(2\pi|z|)$ on that disk,
but is outside Assumption~\ref{ass:rank-one}.
\end{example}
 \end{example}
\begin{remark}[Random weights]
    In our statements weights are deterministic, but with some effort they can be taken to be random, provided that, almost surely, their realized
arrays satisfy all the relevant hypotheses, with
uniform constants and the required empirical limit $\bar \rho$ is well defined. 
    An extension to unbounded weights requires additional
assumptions and a truncation argument controlling the
small singular values and logarithmic potentials. We don't pursue this direction.
   Also note that if the product $w_x^+ w_x^-$ is deterministic, the conditions on the weight profile asks weak convergence to $\bar\rho$, in a deterministic sense.
\end{remark}
Concerning the higher rank setting, the best example is given by stochastic block models.
\begin{example}[Stochastic block models]
    Let $n$ be even and let $a>b>0$. If
    \begin{equation}
    p_{x,y}=
        \begin{cases}
            a  \scaling / n& \text{ if $\max\{x,y\}\le n/2+1$ or $\min\{x,y\} \ge n/2$}\\
            b  \scaling / n& \text{ otherwise},
        \end{cases}
    \end{equation}
    the obtained graph falls in our hypotheses. The has parameters $\theta_1= \frac{a+b}2$ and $\theta_2= \frac{a-b}{2}$. The eigenvector corresponding to $\theta_2$ contains information on the community structure, and the asymptotic behavior of its random realization can be studied as in \cite[Theorems 2.5 and 2.6]{CCH}. A plot of the resulting spectrum, with two outliers is given in Figure \ref{fig:test}(B).
\end{example}
\begin{figure}
\centering
\begin{subfigure}[t]{.46\textwidth}
  \centering
  \IfFileExists{sombrero.pdf}{\includegraphics[width=1\linewidth]{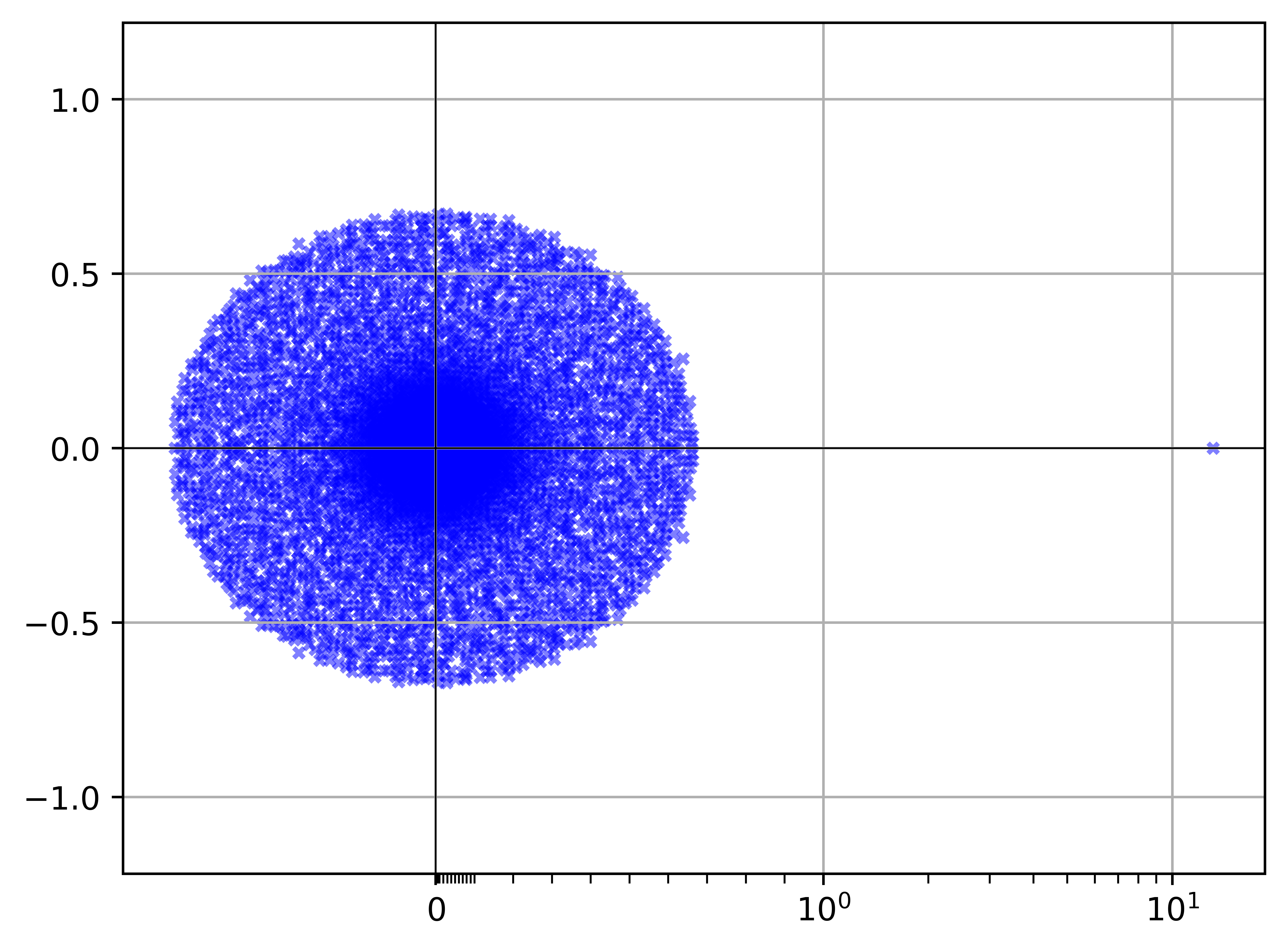}}{\IfFileExists{sombrero.png}{\includegraphics[width=1\linewidth]{sombrero}}{\fbox{\parbox[c][3cm][c]{.9\linewidth}{\centering Original figure unavailable: sombrero}}}}
    \caption{}
\end{subfigure}%
\hfill 
\begin{subfigure}[t]{.46\textwidth}
  \centering
  \IfFileExists{blocks.pdf}{\includegraphics[width=1\linewidth]{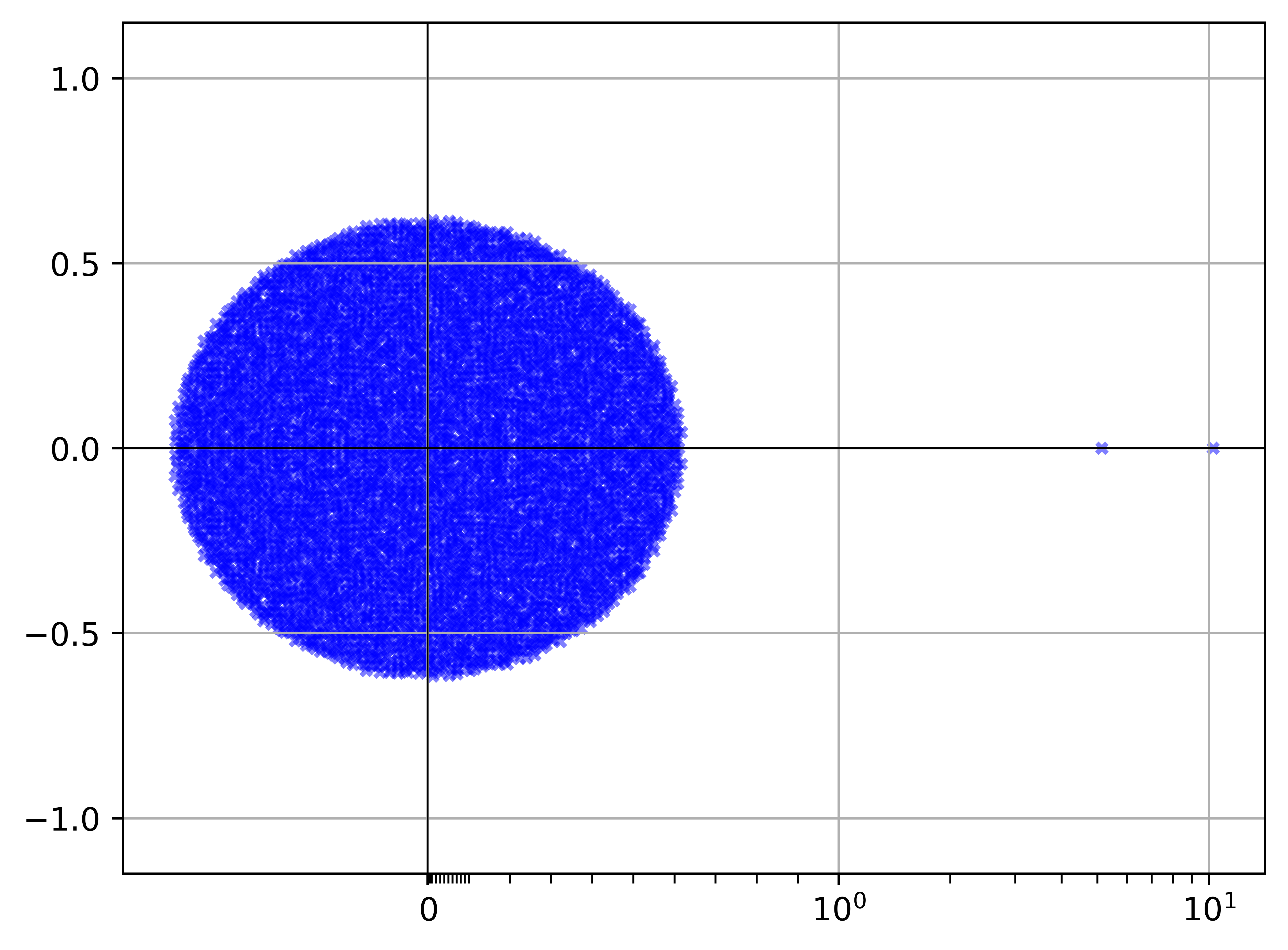}}{\IfFileExists{blocks.png}{\includegraphics[width=1\linewidth]{blocks}}{\fbox{\parbox[c][3cm][c]{.9\linewidth}{\centering Original figure unavailable: blocks}}}}
  \caption{}
\end{subfigure}
\caption{Plot of the spectrum for the adjacency matrix of a Chung--Lu graph with weight distribution as in Example \ref{ex:sombrero} and for a directed block model with two communities. Here $n=6000$ and $\scaling=\log(n)^3$. The scale on the $x$-axis is logarithmic after the threshold $1$, to capture outliers, which are visible on the right.}
\label{fig:test}
\end{figure}

\subsection{Methods and discussion} 

\label{subsec:sp-methods}
\subsubsection*{Bulk}
The analysis of the bulk spectrum follows the so-called 
Girko  Hermitization trick,  which allows the study of a complex empirical spectral distribution via a singular value decomposition using the following equality, valid for any $n\times n$ matrix $M_n$ and $z\in \mathbb{C}$:
\begin{equation} \label{eq:Girko}
\begin{split}
     \frac{1}{n}  \log |{\rm det} (M_n-zI_n)| 
 &=  \frac{1}{n}  \log \sqrt{|{\rm det} (M_n-zI_n)^*(M_n-zI_n)|} \\
 &=  \frac{1}{2n}  \log |{{\rm det} (H_{M_n-zI_n}})|
  =\int_{\R} \log|x|\,d\vartheta_{M_n-zI_n}(x),
  % =\int_0^{+\infty} \log|x|\,d\nu_{M_n-zI_n}(x),
 \end{split}
 \end{equation}
 where
\begin{equation}\label{sym}H_{M_n-zI_n}=\begin{pmatrix} 0 &M_n-zI_n \\ M_n^*-\bar zI_n & 0 \end{pmatrix}\end{equation}
 denotes the Hermitization of $M_n-zI_n$ and 
    $\vartheta_{M_n-zI_n}=\esd{H_{M_n-zI_n}}.$
Notice that this object corresponds to the symmetrized singular value distribution of $M_n-zI_n$. The main task is to identify the limiting distribution of $\vartheta_{M_n-zI_n}$ and to justify the integrability of the logarithm near the origin. This requires quantitative control of the small singular values, and in particular of the least singular value of $M_n-zI_n$. In the present sparse setting, we use estimates from \cite{TV,BasakRudelson19}. Once the logarithmic potential is controlled, a replacement principle, in the spirit of \cite{BasakCookZetouni}, allows us to compare the sparse matrix ensemble with a non-dilute ensemble that is more tractable.

The limiting symmetrized singular value distribution is then identified using estimates and deterministic equivalent results from \cite{CookI}. Finally, the explicit description of the limiting empirical spectral distribution is obtained through asymptotic freeness and $R$-diagonality, following the approach of \cite{HL,GKT}.
 \subsubsection*{Outlier(s)} 
 For the outlier analysis, we view $A_n$ as a random perturbation of its expectation:
$A_n=\mathbb E[A_n]+(A_n-\mathbb E[A_n])$.  In the self-adjoint setting, one can often rely on Weyl-type inequalities or Hoffman--Wielandt estimates to compare the spectra of a matrix and its perturbation. These tools are not directly available in the present non-Hermitian setting. Instead, we use the Bauer--Fike theorem, which applies to diagonalizable matrices and reduces the problem to controlling the spectral norm of the perturbation together with the condition number of the diagonalizing basis.

The required spectral norm estimates are obtained by a high-trace method. This method was developed for undirected random graphs with independent edges in \cite{FK,ChungLu} and for symmetric random matrices in \cite{Vu}, and was later adapted to non-Hermitian matrices in \cite{TV,B,Coste}. In our setting, it provides high-probability bounds on sufficiently high-moments of the spectral norm of the centered matrix $A_n-\mathbb E[A_n]$. The proof relies on combinatorial estimates involving paths, Catalan-type structures, and Dyck words.
While we are interested in the spectral norm, closely related high-moment and non-backtracking arguments were developed in \cite{BBK1} to provide sharp bounds the spectral radius of sparse non-Hermitian matrices.

These estimates show that the noise matrix has spectral norm of order $\sqrt{\scaling}$ with very high probability. Since the nonzero eigenvalues of $\mathbb E[A_n]$ are of order $\scaling$, this separates the deterministic low-rank structure from the random bulk and yields the outlier eigenvalues. In the rank-one Chung--Lu model this gives one outlier, while in the finite-rank model \eqref{eq:rank-m} it gives $r$ outliers.

After establishing the existence and location of the outliers, we study their fluctuations. We follow the strategy of \cite{EKetal,CCH}. The leading eigenvalue is expressed through a fixed-point equation involving a random series. Using concentration estimates and the high-probability bounds obtained in the outlier analysis, this series can be expanded around its deterministic counterpart. Up to lower-order terms, the leading random contribution is a sum of independent random variables. After the appropriate normalization, this sum satisfies a Lindeberg central limit theorem, which yields Gaussian fluctuations. The same idea applies to the rank-$r$ model \eqref{eq:rank-m}, where the fixed-point equation becomes vector-valued and the limiting object is a centered Gaussian vector with an explicit covariance matrix.

\subsubsection*{Open points}
Below we discuss some open points of our analysis:
\begin{itemize}    
\item 
The weak convergence results for the bulk spectrum could likely be strengthened to almost sure convergence with additional work. Since the present paper focuses on identifying the limiting distribution and the outlier behavior, we do not pursue this refinement.
\item 
The regime of bounded average degree, $\scaling=O(1)$, is not covered by our methods. In the homogeneous case this regime has recently been studied in \cite{SahSahasrabudheSawhney25}, where the limiting empirical spectral distribution has a different form and may contain atoms. In the inhomogeneous setting, one should also expect the limiting distribution and the behavior of possible outliers to differ substantially from the diverging-degree regime considered here.
\item 
One could wonder about variants of the model, where higher (still bounded) inhomogeneity is considered. In that case the ESD has been studied, even at a local level, by \cite{XiYangYin-local} and \cite{AltErdosKruger}. In \cite{CookI} a condition called \textit{robust irreducibility} is discussed, for which some connection probabilities can be set to $0$.
\item It would be natural to extend the analysis to Laplacian matrices, in analogy with the undirected inhomogeneous results of \cite{CHHS}. The methods developed here should also give information on the eigenvectors associated with the outliers, similarly to the estimates obtained in \cite[Theorems 2.5 and 2.6]{CCH}. Related questions on eigenvector behavior in directed random graphs have been considered in \cite{BCH,CritErd}.
\item Although our main motivation comes from adjacency matrices of directed random graphs, several parts of the analysis apply more generally to non-Hermitian random matrices with independent entries and controlled inhomogeneity. In particular, one may consider complex-valued finite-rank perturbations, provided that the corresponding outlier eigenvalues are separated in modulus. The Bauer--Fike argument is flexible enough to handle such perturbations, and the fluctuation analysis should extend by decomposing the relevant quantities into real and imaginary parts. Other natural directions include power-law weight profiles, different scalings of the left and right eigenvectors, and finite-rank structures whose rank grows with $n$. These cases require new ideas, especially for the fluctuation theory of the outliers.
\end{itemize}

\section{Bulk analysis}
\label{sec:bulk}

\subsection{Least singular value}
In this subsection we prove a polynomial lower bound for the least singular value of $A_n$ under deterministic perturbations. We first recall the corresponding estimate for sparse matrices with i.i.d.\ entries. We use the formulation given in \cite[Theorem 2.9]{TV}, which is obtained there as an adaptation of \cite[Theorem 2.5]{TV}; the proof is contained in Section 11.1 therein. 
\begin{definition}[$\kappa$-controlled second moment]
    Let $\kappa \ge 1$. A complex random
variable $X$ is said to have $\kappa$-controlled second moment if it satisfies the following upper and lower bounds: $\E[|X|^2] \le \kappa$ and
\[
\E[{\rm Re}(zX - w)^2\one_{\{|X| \le \kappa\}}] \ge \tfrac{1}{\kappa} {\rm Re}(z)^2
\]
for all complex numbers $z, w$.
\end{definition}
\begin{theorem}[{\cite[Theorem 2.9]{TV}}] \label{thm:least-SV-poly-homogeneous}
        Let $A > 1$, $C_1>0$, and $C_2$ be positive constants and let $\alpha \in(0,1)$.  Let $X$ be a random variable with
$C_1$-controlled second moment and let $N_{n}$ be a random matrix of order $n$ with entries 
\begin{equation}
    n_{xy} = {\bf I}_{xy}  X_{xy}, \qquad x,y \in [n],
\end{equation}
where the $({\bf I}_{xy})_{x,y\in [n]}$ and $(X_{xy})_{x,y\in [n]}$ are jointly
independent \iid random variables distributed as a Bernoulli ${\rm Be}(n^{\alpha-1})$ and $X$ respectively. Let $M_n$ be a deterministic matrix of order $n$ satisfying $\|M_n\|\le n^{C_2}$.  Then, there are positive constants $B$ and $C_3$ depending
on $A$, $C_1$, $C_2$, $\alpha$ such that
\begin{equation}
    \p\left(\|(M_n+N_{n})^{-1}\| \ge n^{B}\right) \le C_3 n^{-A}.
\end{equation}
\end{theorem}
In our setting we can state the following.
\begin{theorem} \label{thm:least-SV-poly-inhomogeneous}
        Let $A > 1$ and $C_2>0$ and $\alpha$ be such that Assumptions \ref{ass:rank-one} and \ref{assumption:bulk-poly} hold. Then there are positive constants $B$ and $C_3$ depending on $A$, $C_2$ and $\alpha$, such that for every deterministic matrix $M_n$ with $\|M_n\|\le n^{C_2}$ we have
\begin{equation} \label{eq:thesis-sv}
    \p\left(\sigma_{n}(A_n+M_n) \le n^{-B}\right) \le C_3 n^{-A}.
\end{equation}
\end{theorem}
\begin{proof}
    Choose a positive constant $C$, such that $C> \max_{x,y \in [n]} w_x^+w_y^-$. For every $x,y \in [n]$, we can write
    \begin{equation}
        a_{xy} \sim {\rm \bf I}_{xy} \tilde a_{xy}, 
    \end{equation}
    where ${\rm \bf I}_{xy}$ and $\tilde a_{xy}$ are independent Bernoulli random variables of parameters $$q_n=\tfrac{C\scaling}{\w} \qquad{\text{and}}\qquad \tilde p_{x,y}=\frac{w_x^+w_y^-}{C}$$
    respectively. Since $\w\asymp n$ and $\scaling \sim n^\alpha$, we have $q_n\asymp n^{\alpha-1}$. Moreover, by Assumption \ref{ass:rank-one}, the parameters $\tilde p_{x,y}$ are uniformly bounded away from $0$ and $1$. Let $N_{n}$ be the random matrix with such entries and then we are in the setting of Remark 2.8 in \cite{TV}, which generalizes Theorem \ref{thm:least-SV-poly-homogeneous}. It remains only to verify the domination and controlled moment conditions for the variables $\tilde a_{xy}$.
    \begin{enumerate}
        \item[(i)] there exists a dominating variable $\tilde a$, in the Fourier analytic sense, that is, for every $x,y \in [n]$ it holds 
        \begin{equation}
            \left|\E\left[e^{2\pi i {\rm Re}(\tilde a_{xy}\xi) }\right]\right| \le  \left|\E\left[e^{2\pi i  {\rm Re}(\tilde a\xi) }\right]\right| , \qquad \forall \xi \in \mathbb{C}.
        \end{equation}
        \item[(ii)] there exists $\kappa \ge 1$ such that $\E[|\tilde a|^2] \le \kappa$ and, for any $z , w \in \mathbb{C}$,
        \begin{equation}
        \E[{\rm Re}(z\tilde a-w)^2 \one_{\{|\tilde a |\le \kappa\}}]\ge \frac{1}{\kappa} {\rm Re}(z)^2.
        \end{equation}
    \end{enumerate}
    To check the first condition observe that for every $\xi = \xi_1+i\xi_2$, with $\xi_1,\xi_2\in \R$ it holds
    \begin{equation}
    \begin{split}
        \left|\E\left[e^{2\pi i {\rm Re}(\tilde a_{xy}\xi) }\right]\right|^2 
        & = \big|1-\tilde p_{x,y}+\tilde p_{x,y}e^{2\pi i \xi_1 }\big|^2 \\
        & = \big(1+\tilde p_{x,y}(\cos(2\pi  \xi_1)-1)\big)^2+\tilde p_{x,y}^2\sin^2(2\pi  \xi_1)\\
        &= 1 + 2\tilde p_{x,y}(\cos(2\pi  \xi_1)-1)\big)+2 \tilde p_{x,y}^2- 2\tilde p_{x,y}^2\cos(2\pi  \xi_1)\\
        &=1- 2\tilde p_{x,y}(1-\tilde p_{x,y}) (1- \cos(2\pi \xi_1). 
            \end{split}
    \end{equation}
    Since $\tilde p_{x,y}(1-\tilde p_{x,y})$ is uniformly bounded from below, we may choose a Bernoulli random variable $\tilde a$ with a fixed parameter $q\in (0,1)$ such that
    $$q(1-q)\le \inf_{x,y} \tilde p_{x,y}(1-\tilde p_{x,y}).$$
    %which reaches its maximum over $x$ and $y$ at some $\tilde p \asymp \scaling/n
    % =\tilde p_{\min}= \min_{x,y\in[n]}\tilde p_{x,y}$
    Hence the first condition holds. 
    
    Let us verify the second condition. Let $z = z_1+iz_2$ and $w=w_1+i w_2$. For any $\kappa\ge 1$,
    \begin{equation}
    \begin{split}
        \E[{\rm Re}(z\tilde a-w)^2 \one_{|\tilde a \le \kappa|}]
        &=\E[{\rm Re}(z\tilde a-w)^2 ] 
        = \E[(z_1\tilde a-w_1)^2] \\
        &= \E[\big(z_1(\tilde a -\E[\tilde a])+(z_1\E[\tilde a]-w_1)\big)^2] \\
        &=\E[\big(z_1(\tilde a -\E[\tilde a])\big)^2]+\E[(z_1\E[\tilde a]-w_1)^2]\\ 
        &\ge \E[\big(z_1(\tilde a -\E[\tilde a])\big)^2] 
        = z_1^2 \var(\tilde a)
        =  q(1-q) {\rm Re}(z)^2.
        \end{split}
    \end{equation}
    Choosing $\kappa = (q(1-q))^{-1} \ge q^{-1} \ge 1$, condition (ii) is satisfied.
    Then, by \cite[Remark 2.8]{TV}, it follows that
    \begin{equation}
    \p\left(\|(A_n+M_n)^{-1}\| \ge n^{B}\right) \le C_3 n^{-A},
\end{equation}
which is equivalent to \eqref{eq:thesis-sv}.
\end{proof}

\subsection{Intermediate singular values}
We can prove the following.
This part follows \cite[Lemma 3.5]{BCC14}, which in turn adapts \cite[Proposition 5.1]{TVK} to the sparse setting. 
\begin{theorem} \label{thm:intermediate-SVs} Fix $R>0$. There exist constants $c$ and $C$ depending only on $R$ and on the constants in Assumption \ref{ass:rank-one} such that the following holds: if $\psi_n$ is integer-valued, such that $\psi_n \gg 1$, $\psi_n <
n$ and $\scaling \psi_n/n \ge C \log(n)$, then for any $z \in \mathcal{B}(0,R)$ it holds
\begin{equation} \label{eq:thesis}
    \p\left( 
    \bigcup^{n-1}_{i=3\psi_n}
\left\{\sigma_{n-i}\left(\tfrac{1}{\sqrt{\scaling}}{{A}_n}-zI_n\right)\le c\frac{i}{n}\right\}\right)\le\frac{4}{n^2}.
\end{equation}
\end{theorem}
\begin{proof}
Take $i \in \{3\psi_n ,\dots,n-1\}$ and
    consider the matrix ${A}_{n,m}^z$ obtained taking the first $m=n-\lfloor i/2\rfloor$ rows of ${A}_n-\sqrt{\scaling}zI_n$. By the min-max characterization of eigenvalues, we have $\sigma_{n-i}' \le \sigma_{n-i}$, where $\sigma_j'$ denotes singular values of ${A}_{n,m}^z$, and $\sigma_j$ denotes those of $A_n-\sqrt{s_n} zI_n$. Let $R_j$ be the $j$-th row of ${A}_{n,m}^z$ and let $H_j$ be the subspace spanned by the other rows. By \cite[Lemmata A.1 and A.4]{TVK} we have
    \begin{equation}
        \sum_{j = 1}^{m} (\sigma_{j}')^{-2} =  \sum_{j = 1}^{m}  \frac{1}{({\rm dist} (R_j,H_j)^2}.
    \end{equation}
    In particular, 
    \begin{equation} \label{eq:sing}
        \frac{i}{2n}\frac{1}{\sigma_{n-i}^{2}} \le \sum_{j = 1}^{m}  \frac{1}{n}\frac{1}{({\rm dist} (R_j,H_j)^2},
    \end{equation}
    The coordinates of $R_j$ can be taken to be centered, up to increasing at most by $1$ the dimension of $H$, which reduces the distance.
    For any $j$, if $\Pi^{(j)}$ denotes the projector onto the subspace orthogonal to $H_j$ it holds
    \begin{equation}
        {\rm dist}(R_j,H_j)^2=\sum_{i=1}^{n} |(\Pi^{(j)}R_j)_i|^2 = \sum_{k=1}^n\sum_{i=1}^n \bar R_{j,k}\Pi^{(j)}_{ki}R_{j,i}, 
    \end{equation}
   so that, since for $j \neq k$, $\E[|R_{j,k}|^2]=\E[(a_{jk}-p_{j,k})^2]$, by independence, taking expectations it holds
    \begin{equation}
    \begin{split}
        \E\left[{\rm dist}(R_j,H_j)^2\mid H_j\right]
        &=\sum_{k=1}^n \E[ (a_{jk}-p_{j,k})^2]\Pi^{(j)}_{kk}=
        \sum_{k=1}^n p_{j,k}(1-p_{j,k}) \Pi^{(j)}_{kk}\asymp\frac{\scaling}{n} (n-{\rm dim}(H_j)),
        \end{split}
    \end{equation}
    where we used that $ \sum_{k=1}^n  \Pi^{(j)}_{kk} = n-{\rm dim}(H_j)$ and that the other factors are uniformly bounded from above and below. 
    % Calling $Z={\rm dist}(R_j,H_j)$, and $d={\rm dim}(H_j)$, w
   Since $R_j\mapsto \mathrm{dist} (R_j, H_j)$ is convex and $1$-Lipschitz, by Talagrand's concentration inequality, we get for any $r\ge 0$
    \begin{equation}
        \p\left(|{\rm dist}(R_j,H_j)-M({\rm dist}(R_j,H_j))| \ge r\right) \le 4e^{-r^2/8},
    \end{equation}
    where $M(\cdot)$ denotes the median. This strong concentration implies that there exists a sufficiently small $\eps>0$ and a uniform constant $c$ such that
    \begin{equation} \label{eq:dist}
        \p\left({\rm dist}(R_j,H_j) \le \eps \sqrt{\tfrac{\scaling}{n}(n-{\rm dim}(H_j))}\right)
     \le 4 e^{-\eps^2c\tfrac{\scaling}{n}(n-{\rm dim}(H_j))},
    \end{equation}
    and since ${\rm dim} (H_j) \le n-i/2+1\le  n-\psi_n+1$, we have $n-{\rm dim}(H_j) \ge i/2-1\ge \psi_n -1$. Thanks to the hypothesis on $\psi_n$, the \rhs is at most $4\exp(-3\eps^2c \tfrac{\scaling\psi_n}{n})=4n^{-3\eps cC}$, which is $\le 4n^{-4}$, for $C$ sufficiently large. Combining then \eqref{eq:sing} and \eqref{eq:dist} we have
    \begin{equation}
        \p\left({\tfrac{2n}{i}}\sigma_{n-i}^{2}\le {\eps^2 {\tfrac{\scaling}{n}(\tfrac{i}{2}-1)}}\right) \le 4n^{-3}
    \end{equation}
    which, taking a union bound over $i$ implies
    \begin{equation}
        \p\left(\bigcup_{i=3\psi_n}^{n-1}\left\{\sigma_{n-i}\le {\eps \sqrt{\tfrac{\scaling}{n}(\tfrac{i}{2}-1)\tfrac{i}{2n}}}\right\}\right)\le \frac{4}{n^2},
    \end{equation}
     which is equivalent to \eqref{eq:thesis}. 
\end{proof}

\subsection{Weak convergence for Hermitized matrix}
\label{sec:weak-conv}
In this subsection, for a fixed $z$ in a bounded domain $\mathbb{D}$, we identify the limiting ESD of the Hermitization of $A_n/\sqrt{\scaling}- zI_n$. We recall that the symbol $\vartheta$ is used to denote the ESD of the Hermitized matrix.

    In analogy with \cite[Sec. 10]{BasakRudelson19}, our domain is $$\mathbb{D}=\mathbb{D}_{R,\eps} = \{z \in \mathcal{B}(0,R(1-\eps)) : {\rm Im}(z)\ge \eps\},$$ where $R>0$ is the radius corresponding to the support of the limit (which has a radial density), and ${\rm Im}(z)\ge \eps$ is to apply Theorem \ref{thm:least-SV-sparse} in the very sparse case. Since $\eps\in[0,1]$ can be arbitrarily small, this will give the result on the full support.

\begin{theorem} \label{thm:weak-conv}
For $z \in \mathbb{D}$, $\vartheta_{\tfrac{1}{\sqrt{\scaling}}{A}_n-zI_n}$ converges weakly in probability to a limit $\thetalimit$.
\end{theorem}
Before attempting the proof, we first replace ${A}_n/\sqrt{\scaling}$ by a centered matrix with the corresponding separable variance profile.

\begin{lemma}[Centering and variance correction] 
\label{lem:centering}  
    Let $A_n^0=(a_{xy}^0)_{x,y\in [n]}$ be defined by
    \begin{equation} \label{eq:a0}
        a_{xy}^0 = \frac{ \left(a_{xy}- p_{xy}\right)}{\sqrt{\scaling}}\frac{1}{\sqrt{1-p_{x,y}}}.
    \end{equation}
    Then, if $\mathcal{L}$ denotes the 
Lévy-Prokhorov metric,  $$\mathcal{L}\left(\vartheta_{\tfrac{1}{\sqrt{\scaling}}{A}_n -zI_n}, \vartheta_{{A}_n^0 -zI_n}\right)\pconv 0.$$
    In particular, the two sequences have the same weak limit in probability. 
\end{lemma}
\begin{proof}
    For any two normal matrices $N_n$ and $N_n'$ the Hoffman-Wielandt inequality (\cite{BS}) gives 
\begin{equation} \label{eq:Hoffman-Wielandt}
    \mathcal{L}^3(\mu_{N_n}, \mu_{N'_n}) \le \frac1n \trace[(N_n - N'_n)(N_n - N'_n)^*].
\end{equation}
Then applying this to the Hermitizations, we obtain
 \begin{equation}
    \begin{split}
        \E\Big[\mathcal{L}^3(\vartheta_{\tfrac{1}{\sqrt{\scaling}}{A}_n -zI_n},\vartheta_{{A}_n^0 -zI_n})\Big]
        &
        \le \frac1{2n} \E\trace \Big(\Big( H_{\tfrac{1}{\sqrt{\scaling}}{A}_n}-H_{{A}_n^0}\Big)^2\Big)\\
        &=\frac{1}{n} \sum_{x,y \in [n]} \E[|\tfrac{1}{\sqrt{\scaling}}a_{xy}-a^0_{xy}|^2]\\
        &=\frac{1}{n} \sum_{x,y \in [n]} \E\left[\frac{1}{\scaling}\left(a_{xy}
        \left(1-\tfrac{1}{\sqrt{1-p_{x,y}}}\right)
        +\tfrac{1}{\sqrt{1-p_{x,y}}}\E[a_{xy}]\right)^2\right] \\
        & \le \frac{1}{n} \sum_{x,y \in [n]}\frac{1}{\scaling}\left( \E\left[a^2_{xy}\right]
        \left(1-\tfrac{1}{\sqrt{1-p_{x,y}}}\right)^2
        +\tfrac{1}{{1-p_{x,y}}}\E[a_{xy}]^2\right)\\
        &\le \frac1{n} \sum_{x,y \in[n]}
        \frac{1}{\scaling}\left( p_{x,y}
        \left(1-\tfrac{1}{\sqrt{1-p_{x,y}}}\right)^2
        +\frac{p_{x,y}^2}{1-p_{x,y}}\right).\\
        % &= \frac1{n\scaling} \sum_{x,y \in[n]} p_{x,y}^2+\frac1{n\scaling} \sum_{x,y \in[n]} p_{x,y} \left(\sqrt{1-p_{x,y}}-1 \right)^2 \asymp \frac{\scaling}{n}+\frac{\scaling^2}{n^2}.
        \end{split}
    \end{equation}
    Using that $1-\tfrac{1}{\sqrt{1-x}}\sim -\tfrac{x}{2}$ for $x\to 0$, we get that the above is
\begin{equation}
        \lesssim \frac{1}{n} \sum_{x,y \in [n]}\frac{1}{\scaling}\left( \frac{p_{x,y}^3}{4}+p_{x,y}^2\right) \le C\left(\frac{\scaling^2}{n^2}+\frac{\scaling}{n}\right)=o(1).
    \end{equation}
    The claim follows from Markov inequality. 
\end{proof}
\begin{lemma}[Concentration around the mean] \label{lemma:concentration}
    Let $M_n$ be a matrix with independent rows or columns. Then, for any $f:\mathbb{R}\to \mathbb{R}$ vanishing at infinity, with $\|f\|_{TV}\le 1$ and every $t\ge 0$, it holds
    \begin{equation}
        \p\left(\left|\int f(x) \, d\vartheta_{M_n-zI_n}(x)-\E\int f(x) \, d\vartheta_{M_n-zI_n}(x)\right|\ge t\right) \le 2e^{-2nt^2}.
    \end{equation}
\end{lemma}
\begin{proof}
     By {\cite[Lemma 4.18]{BC}}, for any function $g:\mathbb{R}^+ \to \mathbb{R}$ vanishing at infinity, with $\|g\|_{TV}\le 1$,
        \begin{equation}
        \p\left(\left|\int g(x) \, d\nu_{M_n-zI_n}(x)-\E\int g(x) \, d\nu_{M_n-zI_n}(x)\right|\ge t\right) \le 2e^{-2nt^2}.
    \end{equation}
    Writing $g(x)=\frac{f(x)+f(-x)}{2}$, we have it for $\vartheta_{M_n-zI_n}$ as well. The complex shift $z$ does not affect the argument; see also \cite{BCC14}. 
\end{proof}
% \subsubsection{Invariance principle}

Now, we can state the following invariance principle, which allows to substitute the Bernoulli entries with Gaussian ones.
Let $G_n^{\mathbb{R}}$ denote a matrix with \iid centered real Gaussian random variables with variance $1$, and let
 \begin{equation}
    {A}_n^g=\tfrac{1}{\sqrt{n}}(\bar{W}_n^+)^{1/2}G_n^{\mathbb{R}} (\bar{W}_n^-)^{1/2}.
\end{equation}
\begin{lemma}[Replacement principle] \label{lemma:Lindeberg}
    For any $z \in \mathbb{C}$ and $w\in \mathbb{C}^+$ it holds 
        \begin{equation}
        \Big|\E \mathcal{G}_{\vartheta_{{A}_n^0-zI_n}}(w) - \E \mathcal{G}_{\vartheta_{{A}_n^g-zI_n}} (w)\Big| = O\left(\frac{1}{\sqrt{\scaling}\ {\rm Im}(w)^4} 
        % \tfrac{1}{}
        \right).
        % \left(1+\frac{1}{(n {\rm Im}(w))^2}\right),
    \end{equation}
\end{lemma}
\begin{proof}[Proof of Lemma \ref{lemma:Lindeberg}]
   The entries of the two matrices are centered and  satisfy for any $x,y \in [n]$,
    \begin{equation} \label{eq:lindeberg}
    % \begin{split}
        \E[(a^0_{xy})^2] 
        =  \frac{1}{ {\scaling}(1-p_{x,y})}\,\var[a_{xy}]
        =  \frac{1}{ {\scaling}} \,p_{x,y}
        = \frac{1}{n}\bar{w}_x^{+} \bar{w}_y^{-}.
    % \end{split}
    \end{equation}
    This is the same variance profile as the entries of $A_n^g$. Moreover by Assumption \ref{ass:rank-one},
    $$ \max_{x,y\in [n]}\E[| a_{xy}^0|^3]= \mathrm{O}\Big(\frac{1}{n\sqrt{s_n}}\Big), \qquad \max_{x,y\in [n]}\E[| a_{xy}^g|^3]= \mathrm{O}\Big(\frac{1}{n^{3/2}}\Big) $$
    and the corresponding Gaussian third-moment contribution is of the same order. We apply the Lindeberg replacement principle from \cite{Chatterjee2} to the real and imaginary parts of the normalized resolvent trace
    $$U\mapsto \mathcal{G}_{\vartheta_{U-zI_n}}.$$

    The required derivative bounds are the standard resolvent bounds and are same as in \cite[Lemma 8.2]{Cook}. 
    Our scaling is different form the one in (9.3) therein (notice that their $X=M(B)$ corresponds to our $\sqrt{n}A^0_n$).
    A bound as in (9.9) therein, with $N=n^2$, 
    \begin{equation}
        \gamma_3 = {n^{3/2}}\max_{x\in [n]} \{\E[|a_{xy}^0|^3],\E[|a_{xy}^g|^3|]\}= O\big(\sqrt{{{n}}/{\scaling}}\big),
    \end{equation}
    and $\lambda_3 =O(1/(n^{5/2}\Im(w)^{4}))$ (the rescaling is there hidden in the definition of ${\bf H}_z(M)$ and hence the bound stays the same)
    proves the claim.
\end{proof}

% \subsubsection{Integrability of the weak limit}

By \cite{CookI}, for every $z\in \mathbb{D}$, the measures $\vartheta_{{A}_n^g-zI_n}$ weakly converge to a deterministic limit. Then to prove Theorem \ref{thm:weak-conv}, it suffices to apply Lemma \ref{lemma:concentration} and Lemma \ref{lemma:Lindeberg}. 

\begin{proof}[Proof of Theorem \ref{thm:weak-conv}]
    In order to prove the weak convergence, we only need to show that
    \begin{equation}
        \int_\mathbb{R} f(x) \, d\vartheta_{{A}_n^0-zI_n}(x) -       \int_\mathbb{R} f(x) \, d\vartheta_{{A}_n^g-zI_n}(x) \pconv 0,
    \end{equation}
    for any continuous function $f:\R\to \R$ with compact support.  Indeed the probability distributions in the game have with high probability compact support. Applying the concentration result in Lemma \ref{lemma:concentration} we have
    \begin{equation}
        \int_\mathbb{R} f(x) \, d\vartheta_{{A}_n^0-zI_n}(x)-\E\int_\mathbb{R} f(x) \, d\vartheta_{{A}_n^0-zI_n}(x) \pconv 0,
    \end{equation}
    (and the same for ${A}_n^g-zI_n$)
    and by the invariance principle Lemma \ref{lemma:Lindeberg}, we have 
        \begin{equation}
        \E\int_\mathbb{R} f(x) \, d\vartheta_{{A}_n^0-zI_n}(x)-\E\int_\mathbb{R} f(x) \, d\vartheta_{{A}_n^g-zI_n}(x) \conv 0.
    \end{equation}
    This concludes the proof.
\end{proof}

\subsection{Proof of the bulk result}

We are almost ready to prove the main result.
% \subsection{Proof of the main results}
We employ the following replacement principle, which requires some integrability conditions.
\begin{lemma}[{\cite[Lemma 8.1]{BasakCookZetouni}}] \label{thm:replacement}
Assume that $\matrixx$ and $\matrixxprime$ are $n \times n$ random matrices such that
\begin{enumerate}
    \item $\frac{1}{n}\|\matrixx\|_2^2 + \frac{1}{n}\|\matrixxprime\|_2^2$ is bounded in probability;
    \item there exists a finite $R$ and a domain $\mathbb{D} \subseteq \mathcal{B}(0,R)\subseteq \mathbb{C}$ such that for almost all ${z} \in \mathbb{D}$,
    \begin{equation}    \label{eq:log-determinant}
        \frac1n \log\left|{\rm det}\left(\matrixx-{z} I_n\right)\right|-\frac1n \log\left|{\rm det}\left(\matrixxprime-{z} I_n\right)\right| \pconv 0,
    \end{equation}
\end{enumerate}
Then for every $f \in C^2_c(\mathbb{C})$  supported on $\mathbb{D}$, 
\begin{equation}
    \int f(z') \,d\mu_{\matrixx}(z')- \int f(z')\, 
    d\mu_{\matrixxprime}(z') \pconv 0.
\end{equation}
\end{lemma}

Thanks to \eqref{eq:Girko}, to show that the conditions of Lemma \ref{thm:replacement} are satisfied, we have to show that the difference 
$$\int_{\R} \log|x|\,d\vartheta_{\frac{1}{\sqrt{\scaling}A_n}-zI_n}(x) -\int_{\R} \log|x|\,d\vartheta_{{A}_n^g-zI_n}(x)$$
vanishes, since the first assumption follows by the weak law of large numbers.
Below, we show that the limit $\theta_z$ mentioned in the previous subsection has such good integrability properties and then prove the main statements on the bulk.
\begin{proposition}\label{prop:integrability}
For $z\in\mathbb D_{R,\eps}$ there is a positive constant $C_{|z|}$
such that, for $0<\tau<1$,
\[
\int_{-\tau}^{\tau}|\log|x||\,d\vartheta_z(x)
\le C_{|z|}\tau(1+|\log\tau|).
\]
In particular the bound is $O(\tau|\log\tau|)$ for
$0<\tau\le e^{-1}$. Moreover, the measure $\vartheta_z$ has
compact support.
\end{proposition}
 \begin{proof}
The bound follows from an estimate on the Stieltjes transform of $\thetalimit$. 
% We now prove the logarithmic integrability estimate. 
We obtain it using the deterministic-equivalent results of \cite{CookI} for non-Hermitian random matrices with a variance profile. In the present case, the standard-deviation profile of $A_n^g$ is
\[
        \sigma_{xy}^{(n)}
        =
        (\bar w_x^+)^{1/2}(\bar w_y^-)^{1/2}.
\]
Again by Assumption \ref{ass:rank-one}, this profile is uniformly bounded above and below by positive constants. Therefore the boundedness assumption and the lower-bound assumption on the variance profile in \cite{CookI} are satisfied. In particular, the associated admissibility condition holds. Let $\check\nu_{n,z}$ denote the deterministic equivalent, in the notation of \cite{CookI}, for the symmetrized singular value distribution of $A_n^g-zI_n$. The Schwinger--Dyson equations in \cite[equation (2.23)]{CookI}, evaluated at the spectral parameter $i\eta$, imply that
\[
        {\rm Im}\,\mathcal G_{\check\nu_{n,z}}(i\eta)
        =
        \frac1n\sum_{x=1}^n r_x(\eta),
\]
where $(r_x)_{x\in[n]}$ is the solution of the corresponding regularized master equations \cite[Equation 2.7]{CookI}. By the admissibility estimate of \cite[Assumption {\bf A2}]{CookI}, for every $z\in\mathbb D_{R,\eps}$ there exists a constant $C_{|z|}<\infty$ such that
\begin{equation}
\label{eq:Cook-Stieltjes-bound}
        |{\rm Im}\,\mathcal G_{\check\nu_{n,z}}(i\eta)|
        \le C_{|z|}
        \qquad \text{for all } \eta\in(0,1),
\end{equation}
uniformly in $n$. Since $\vartheta_{A_n^g-zI_n}$ converges to $\thetalimit$ and the deterministic equivalents have the same limit, the bound \eqref{eq:Cook-Stieltjes-bound} passes to the limit. Thus
\begin{equation}
\label{eq:limit-Stieltjes-bound}
        |{\rm Im}\,\mathcal G_{\thetalimit}(i\eta)|
        \le C_{|z|}
        \qquad \text{for all } \eta\in(0,1).
\end{equation}
Using
\[
        {\rm Im}\,\mathcal G_{\thetalimit}(i\eta)
        =
       - \int_{\mathbb R}\frac{\eta}{x^2+\eta^2}\,d\thetalimit(x),
\]
we obtain
\[
        \thetalimit((-\eta,\eta))
        \le
        2\eta\,|{\rm Im}\,\mathcal G_{\thetalimit}(i\eta)|
        \le
        C_{|z|}\eta .
\]
Finally, by integration by parts,
\[
\begin{split}
    \int_{-\tau}^{\tau} |\log |x||\,d\thetalimit(x)
    &\le
    |\log\tau|\,\thetalimit((-\tau,\tau))
    +
    \int_0^\tau \frac{\thetalimit((-t,t))}{t}\,dt  \\
    &\le
    C_{|z|}\tau|\log\tau|
    +
    C_{|z|}\tau ,
\end{split}
\]
 and the desired estimate follows after increasing the constant $C_{|z|}$.
 The compact support follows by observing that $\thetalimit$ is the limit of a polynomial function of uniformly compactly supported objects and the boundedness of the variance profile.
\end{proof}

\begin{proof}[{Proof of Theorem \ref{thm:bulk}}]
We will apply Lemma \ref{thm:replacement} with the choices $M_n = \frac{1}{\sqrt{\scaling}}A_n$ and $M'_n={A}_n^g$.
Using \eqref{eq:Girko}, we take $ 0< \tau <K <+\infty$ and we decompose the range of integration in three parts. We start with the contribution on $(-\tau,\tau)$.
\begin{equation}
\begin{split}
 \int_{-\tau}^\tau |\log|x||\,d&\vartheta_{\tfrac{1}{\sqrt{\scaling}}{A}_n-zI_n}(x) = \frac{1}{n}\sum_{x=1}^{n-3\psi_n} |\log(\sigma_{x}(\tfrac{1}{\sqrt{\scaling}}{A}_n-zI_n)| \one(\sigma_{x}(\tfrac{1}{\sqrt{\scaling}}{A}_n-zI_n)\le \tau)\\
 &+\frac{1}{n}\sum_{x=n-3\psi_n+1}^{n} \!\!\!\!|\log(\sigma_{x}(\tfrac{1}{\sqrt{\scaling}}{A}_n-zI_n)|\one(\sigma_{x}(\tfrac{1}{\sqrt{\scaling}}{A}_n-zI_n)\le \tau), 
 \end{split}\end{equation}
 Reasoning as in \cite[p. 2410]{BasakRudelson19}, the second term is bounded using Theorem \ref{thm:least-SV-poly-inhomogeneous} (or Section \ref{sec:sparse} depending on the scaling factor $\scaling$). 
 The first addend is estimated using Theorem \ref{thm:intermediate-SVs}. In particular, fixing $\delta \in (0,1)$ and choosing $\tau = c \delta$ for $c$ as in Theorem \ref{thm:intermediate-SVs}, we can have that it is bounded by
\begin{equation}
    \kappa(\delta)=2\delta \log(1/c)-2\int_0^\delta \log(x) \,dx,
\end{equation}
which vanishes as $\delta$ decreases. On the intervals $(-K,-\tau)$ and $(\tau,K)$, $\log|x|$ is bounded continuous, hence, by Theorem \ref{thm:weak-conv} we have convergence to the same integral w.r.t.\ $\thetalimit$, which is finite on $(-K,K)$ since the singularity around $0$ is controlled by Proposition \ref{prop:integrability}. Since $\thetalimit$ has compact support, on the intervals $(-K,K)^\complement$, the limit integral vanishes.
% and \textcolor{red}{by Proposition \ref{prop:prop-W} (or \cite[Theorem 1.7]{BasakRudelson-invertibility} in the sparser case)} the same holds with high probability for the integral w.r.t.\ $\vartheta_{\tfrac{1}{\sqrt{\scaling}}{A}_n-zI_n}$, for $K$ sufficiently large. 
If $K>\sqrt e$, since $x\mapsto \log x/x^2$ is decreasing on $[K,\infty)$, we also have
\begin{equation}
     \int_{|x|>K}\log|x|\,d\vartheta_{M_n}(x)
 = 
 \int_{|x|>K}\frac{\log|x|}{|x|^2} |x|^2\,d\vartheta_{M_n}(x)
 \le 
 \frac{\log K}{K^2}\frac1n\|M_n\|_F^2,
\end{equation}
where $M_n=\tfrac{1}{\sqrt{\scaling}}{A}_n-zI_n$. Taking expectations, for $z$ in a bounded set,
\[
 \mathbb E\!\left[\frac1n\|M\|_F^2\right]
 \le 2\frac{1}{ns_n}\sum_{x,y}p_{xy}+2|z|^2
 =2\frac wn+2|z|^2=O(1).
\]
Then, the logarithmic tails at infinity can be made arbitrarily small in probability, taking $K$ arbitrarily large.
The same estimate holds for $A_n^{g}-zI_n$.
The same kind of estimates can be derived for $A_n^g$.
Then, by the replacement principle, for any function $f \in C^2_c(\mathbb{D}_{R,\eps})$
\begin{equation}
    \int_{\mathbb{C}} f(z') d\mu_{\tfrac{1}{\sqrt{\scaling}}A_n-zI_n}(z')- \int_{\mathbb{C}} f(z') d\mu_{{A}_n^g-zI_n}(z') \pconv 0,
\end{equation}
in probability.  Since the limit distribution $\mu$ has a compact support, the same can be obtained for any general $f \in C^2_c(\mathbb{C})$, introducing a suitable cutoff function which is $1$ on $\mathbb{D}_{R,2\eps}$ and has support $\mathbb{D}_{R,\eps}$.
The replacement principle gives convergence of the
unshifted eigenvalue measures upper half-plane. Since the matrices $A_n$ and the real
Gaussian comparison are real, their eigenvalue
measures are invariant under complex conjugation and 
the same convergence holds in the lower half-plane.
The identified radial limit has a planar density,
and assigns zero mass to the real line and to its
support boundary.\\
 We stress that the Gaussian matrix used in Lemma \ref{lemma:Lindeberg} is real.  
    However, since the deterministic-equivalent result of \cite{CookI} applies to matrices with \iid centred entries of unit variance and the corresponding deterministic equations depend only on the variance profile, we may therefore use 
    % the real Gaussian ensemble for the replacement argument and 
    the complex Ginibre ensemble for the identification of the limiting ESD. This concludes the proof.
\end{proof}

\begin{proof}[Proof of Corollary \ref{cor:bulk}]
We first observe that 
\[
(\bar W_n^+)^{-1/2}
[(\bar W_n^+)^{1/2}G_n(\bar W_n^-)^{1/2}]
(\bar W_n^+)^{1/2}=G_n\bar W_n,
\]
so that $\bar W_n^{+}G_n\bar W_n^{-}/\sqrt{n}$ and $G_n\bar W_n/\sqrt{n}$ are similar and have the same limiting distribution. Moreover
$\bar{W}_n$ and $G_n/\sqrt{n}$ are asymptotically free by \cite[Theorem 4.3]{HP}.
Then $G_n/\sqrt{n}$ and
$\bar W_n$ converge in $*$-distribution to free
operators $\mathfrak c$ and $\bar{\mathfrak d}$,
where $\mathfrak c$ is circular. Given the
logarithmic-integrability bounds discussed above, this identifies $\mu$ with
the Brown measure of $\mathfrak c\bar{\mathfrak d}$.

Since $\mathfrak{c}$ is $R$-diagonal and $\bar{\mathfrak{d}}$ is positive definite,  thanks to \cite[Theorem 9.7]{speicher2025lecturenotesfreeprobability},  $\mathfrak{c}\bar{\mathfrak{d}}$ is $R$-diagonal. As a consequence, using \cite[Theorem 9.8]{speicher2025lecturenotesfreeprobability}, we can write the polar decomposition $\mathfrak{c}\bar{\mathfrak{d}}=\mathfrak{u} \mathfrak{h}$, where $\mathfrak{u}$ is unitary Haar-distributed and $\mathfrak{h}=|\mathfrak{c}\bar{\mathfrak{d}}|=\sqrt{(\mathfrak{c}\bar{\mathfrak{d}})^*\mathfrak{c}\bar{\mathfrak{d}}}=\sqrt{\bar{\mathfrak{d}}\mathfrak{c}^*\mathfrak{c}\bar{\mathfrak{d}}}$ is the modulus of $\bar{\mathfrak{d}}\mathfrak{c}$. Since $\mathfrak{h}^2=\sqrt{\bar{\mathfrak{d}}^2}\mathfrak{c}^*\mathfrak{c}\sqrt{\bar{\mathfrak{d}}^2}$, by freeness, again by \cite[Theorem 4.3]{HP} we obtain
% \begin{equation}
    $\mu_{\mathfrak{h}^2}=\mu_{{\bar{\mathfrak{d}}}^2} \boxtimes m$,
% \end{equation}
so that 
\begin{equation}
    \mathcal{S}_{\mu_{\mathfrak{h}^2}}(z) = \frac{\mathcal{S}_{\mu_{{\bar{\mathfrak{d}}}^2}}(z)}{z+1}.
\end{equation}
This means that using Haagerup--Larsen theory
\cite[Theorem 4.4]{HL} (see also \cite[Sect 11.6, Theorem 8]{mingo}), we have
  \begin{equation}
        \mu\left(\mathcal{B}(0,F(t))\right) = t, \qquad \text{where }\quad  F(t)=\tfrac{1}{\sqrt{\mathcal{S}_{\mu_{\mathfrak{h}^2}}(t-1)}},
    \end{equation}
and the conclusion follows. The second part follows by \cite[Corollary 4.5]{HL}, noticing that the support of the radial density is given by an annulus with lower and upper radii $\|\mathfrak{h}^{-1}\|_2^{-1}=0$ and $\|\mathfrak{h}\|_2$.
\end{proof}

\subsection{Asymptotic freeness}

In this subsection we show asymptotic freeness of the $2n \times 2n$ matrices arising from the symmetrization process and then prove Theorem \ref{thm:SV}.

\begin{proposition}
\label{prop:freeness}
    Let $G_n$ be a bi-unitary invariant matrix and
    $$
    \widetilde{W}_{n}=\begin{pmatrix}  \bar{W}_n^+ & 0 \\ 0 & \bar{W}_n^- \end{pmatrix}, $$
    and assume that the weak limits of $\esd{\bar{W}_n^+}$ and $\esd{\bar{W}_n^-}$coincide with $\tilde \rho$.
    If $\vartheta_{G_n}=\esd{H_{G_n}}$ converges weakly in probability to a compactly supported probability measure $\mu$, and for any $k \in \N$,
    \begin{equation}
        \sup_{n \in \N} \frac 1{2n} \E \trace H_{G_n}^{2k}<+\infty,
    \end{equation} 
    then $H_{G_n}$ and $\widetilde{W}_n$ are asymptotically free.   
\end{proposition}
\begin{proof}
For any given $2n \times 2n$ matrix $M_n$ and $l \in \N$, let
$$\widehat M_n^l = M_n^{l}-\left(\lim_{\nu \to +\infty}\frac{1}{2\nu} \E \left[\trace M_\nu^{l}\right]\right)I_{2n}.$$
We need to show that for any $k \in \N$ and $l_1 ,\dots, l_k \in \N$,
\begin{equation} \label{eq:free}\lim_{n \to +\infty} \E \left[\tfrac 1{2n} \trace \left(\widehat M_{1,n}^{l_1} \cdots \widehat M_{k,n}^{l_k}\right)\right]=0,\end{equation}
where $M_{i,n}$ can be $H_{G_n}$ or $\widetilde{W}_n$ and $M_{i,n}\neq M_{i+1,n}$ for $i=1,\dots,k$.
Using the notion of bi-unitary invariance (\cite[Theorem 3.2]{HP}) we can write $G_n = U_n \Delta_n V_n$, where $U_n$ and $V_n$ are unitary random matrices (with Haar distribution) and $\Delta_n$ contains the SVD of ${G_n}$, taken with uniform random signs. Then,
\begin{equation} \label{eq:blocks}
    \widehat H_{G_n}^l\!= \begin{cases}
    \begin{pmatrix}
        U_n \big[\Delta_n^{2p}-\int_{\R} x^{2p} \,d\mu(x) I_n\big]U^*_n&0\\
        0& V_n \big[\Delta_n^{2p}-\int_{\R} x^{2p} \,d\mu(x) I_n\big] V^*_n
    \end{pmatrix}, & \text{ if } l=2p\\
    \begin{pmatrix}
        0&  U_n \Delta_n^{2p+1} V^*_n\\V_n \Delta_n^{2p+1} U^*_n&0
    \end{pmatrix}, & \text{ if } l=2p+1.
    \end{cases}
\end{equation}
It is not difficult to see that
\begin{equation}
    \tfrac1n\trace {(\Delta_n^l)}-\int x^l \,d\mu(x) \pconv 0.
    \qquad
\end{equation}
The terms involving $\widetilde{W}_n$ are block diagonal with blocks of the form ${(\widetilde{W}_n^\#)^l}-\int x^l \,d\tilde \rho(x) I_{n}$, for some $l \in \N$ and $\# \in\{+,-\}$.
It follows, that the blocks of the matrix in \eqref{eq:free} are given by products of the four non-zero blocks in \eqref{eq:blocks} and the diagonal blocks arising from $\widetilde{W}_n$, arranged in such a way
% each occurrence of $U_n^*$ is followed only by $V_n$, and each occurrence of $V_n^*$ is followed only by $U_n$. 
that each occurrence of $U_n,U^*_n,V_n$, or $V^*_n$ is followed and preceded by a term involving $\Delta_n$, $\bar W_n^+$, or $\bar W_n^-$, so that they are well separated. 
% {this is not true, but the rest should still be correct ($U_n$'s will be always separated from $V_n$'s by the blocks of $H_{\widetilde{W}_n}$)}
Since reasoning as in the proof of \cite[Proposition 5.8]{GKT}, it holds that $\{U_n, U^*_n\}$, $\{V_n, V_n^*\}$ and $\{\Delta_n,\bar{W}_n^+,\bar{W}_n^-\}$ are asymptotically free, we have \eqref{eq:free}. Notice that the proof in \cite{GKT} first conditions on the realization of $\Delta$, so that asymptotic freeness between Haar unitary and deterministic matrices is employed. Then via a subsequence argument the desired claim is achieved.
\end{proof}

\begin{corollary} \label{cor:freeness}
Let $G_n$ be a matrix with \iid complex centered Gaussian entries with variance $1$ and assume that the weak limits of $\esd{\bar{W}_n^+}$ and $\esd{\bar{W}_n^-}$ coincide with $\tilde \rho$. Then $\vartheta_{{A}_n^g}$ converges weakly in probability to $\tilde \rho
    \boxtimes s$, where $s$ is the semicircular distribution.
\end{corollary}
\begin{proof}
    It follows from Proposition \ref{prop:freeness}, noting that the symmetrized version of 
    ${A}_n^g
    % -zI_n
    $ 
    is given by
    \begin{equation}
   \frac{1}{\sqrt{n}} \begin{pmatrix} (\bar{W}_n^+)^{1/2} & 0  \\ 0 & (\bar{W}_n^- )^{1/2} \end{pmatrix}
        \begin{pmatrix} 0 & G_n \\  G_n^*   & 0 \end{pmatrix}
        \begin{pmatrix} (\bar{W}_n^+)^{1/2} & 0  \\ 0 & (\bar{W}_n^- )^{1/2} \end{pmatrix}
        % +\begin{pmatrix}0 & -z I_n \\ -\bar z I_n & 0  \end{pmatrix}
        .
\end{equation}
Moreover, the semicircular distribution $s$ is the symmetrized push-forward, via the function $x \mapsto \sqrt{x}$, of the Marchenko--Pastur distribution of parameter $1$ in \eqref{eq:marchenko-1}.
Since the eigenvalues of $H_{G_n}$ are the symmetrized (i.e., positive and negative) square roots of eigenvalues of $\tfrac 1{n} (G_n)^*G_n$, which are Marchenko--Pastur asymptotically distributed (by standard facts on Wishart matrices, see \cite{mingo}), we get that $\vartheta_{\tfrac{1}{\sqrt{n}}G_n}$ converges weakly in probability to $s$.
\end{proof}

\begin{proof}[Proof of Theorem \ref{thm:SV}]
    Let us first start with proof of the second claim. By Lemmata \ref{lem:centering}, \ref{lemma:concentration}, and \ref{lemma:Lindeberg} 
    it follows that the Hermitizations of
        $
        % \vartheta_
        {\tfrac{1}{\sqrt{\scaling}}A_n}$ and $
        % \vartheta_
       {\tfrac{1}{\sqrt{n}}(\bar{W}_n^+)^{1/2}G_n (\bar{W}_n^-)^{1/2}}$ have ESDs converging weakly in probability to the same limit, which corresponds to the symmetrized limit of the singular value distribution of $
        % \vartheta_
        {\tfrac{1}{\sqrt{\scaling}}A_n}$ and $
        % \vartheta_
        {\tfrac{1}{\sqrt{n}}(\bar{W}_n^+)^{1/2}G_n (\bar{W}_n^-)^{1/2}}$.
        By the asymptotic freeness in Corollary \ref{cor:freeness}, the limit singular value distribution is then $\sqrt{(\tilde \rho \boxtimes s)^2}$ which is equal to $\sqrt{\tilde \rho \boxtimes m \boxtimes\tilde \rho}$ by \cite[Lemma 8]{ArizmendiPerezAbreu}.
 This concludes the proof of the second claim. To prove the first claim, we let     \begin{equation}
    {B}_n^g=\tfrac{1}{\sqrt{n}}(\bar{W}_n)^{\frac12} G_n (\bar{W}_n)^{\frac12},
\end{equation}
where we recall that $\bar{W}_n = \sqrt{\bar{W}_n^+ \bar{W}_n^-}$. If we define $B_n^0$ in a similar fashion as \eqref{eq:a0}, then \eqref{eq:lindeberg} becomes
\begin{equation}
\mathbb E[|b_{xy}^0|^2]
=\frac{p_{xy}}{s_n}
 \sqrt{\frac{w_x^-w_y^+}{w_x^+w_y^-}}
=\frac{\bar w_x\bar w_y}{n}
=\mathbb E\left[\frac{{\bar w_x\bar w_y}}
 { n}|g_{xy}|^2\right].
\end{equation}
    while Proposition \ref{prop:freeness} and Corollary \ref{cor:freeness} still hold for the matrix decomposition
    \begin{equation}
    \frac{1}{\sqrt{n}}\begin{pmatrix} (\bar{W}_n)^{\frac12} & 0  \\ 0 & (\bar{W}_n)^{\frac12}  \end{pmatrix}
        \begin{pmatrix} 0 & G_n \\  G_n^*   & 0 \end{pmatrix}
        \begin{pmatrix} (\bar{W}_n)^{\frac12} & 0  \\ 0 & (\bar{W}_n)^{\frac12}  \end{pmatrix}
        % +\begin{pmatrix}0 & -z I_n \\ -\bar z I_n & 0  \end{pmatrix}
        ,
\end{equation}
so that the limit singular value distribution is in this case $\sqrt{\bar \rho \boxtimes m \boxtimes \bar \rho }$.
\end{proof}

\section{Bulk: sparser setting}
\label{sec:sparse}

In order to complete the characterization of the bulk in the sparser setting where Assumption \ref{assumption:bulk-sparse} holds, it is sufficient to provide a more refined estimate for the least singular value than those used in the proof of Theorem \ref{thm:least-SV-poly-inhomogeneous}.
% \subsection{Least singular value in sparser setting}
In this regard, the following result has been stated and proved for the \iid case.

\begin{theorem}[{\cite[Theorem 11.3]{BasakRudelson19-arxiv}}] \label{thm:least-SV-sparse}
Assume $C_0\log n\le s_n\le n/2$ for a sufficiently
large constant $C_0$.
Fix $R \ge 1$, $\eps \in (0,1]$ and let $D_n$ be a diagonal matrix such that
$\|D_n\| \le R\sqrt{\scaling}$ and ${\rm Im}(D_n)= r'\sqrt{\scaling} I_n$ for some $r'$ with $|r'|\in [\eps R,R]$. 
If $A_n$ is an $n\times n$ random matrix with zero diagonal and \iid Bernoulli off-diagonal entries with parameter $s_n/n$, then there
exist constants $c,C,C'>0$
depending only on $R,\eps$, 
such that for any $\bar \eps > 0$ we have the following:
\begin{equation}
    \p\left(\sigma_{n}(A_n+D_n) \le c \ \bar \eps \exp\left(-C{ \frac{\log(n/\scaling)}{\log(\scaling)}}\right)\sqrt{\frac{\scaling}{n^2}}\right)\le \bar \eps +\frac{C'}{\sqrt{\scaling}}.
\end{equation}
\end{theorem}
For the reader convenience, in the following we are going to provide a proof sketch of this result. Later we will highlight the adaptations needed for our inhomogeneous setting.

The proof of Theorem \ref{thm:least-SV-sparse} is based on the following variational formula 
\begin{equation} \label{eq:sigman}\sigma_{n}(A_n + D_n) = \min_{v \in \mathbb{S}^{n-1}} \|(A_n + D_n)v\|,
\end{equation}
and decomposing $\mathbb{S}^{n-1}$ into subfamilies of vectors, whose definitions are taken, e.g., from \cite{RUDELSON2008600,BasakRudelson-invertibility}.
Here we denote with $\mathbb{S}^{n-1}$ the complex sphere in $\mathbb{C}^n$.
\begin{definition}[Sparse vectors]
    Fix $m < n$. A vector $v \in \mathbb{C}^n$ is said to be $m$-sparse if its support has cardinality at most $m$. The set of $m$-sparse complex vectors is denoted with ${\rm Sparse}(m)$.
\end{definition}
\begin{definition}[Compressible vectors]
    Fix $m < n$ and $\delta>0$. A vector $v \in \mathbb{S}^{n-1}$ is said to be $(m,\delta)$-compressible if there exists $u \in {\rm Sparse}(m)$ such that $\|v-u\|\le \delta$. The set of such vectors is denoted by ${\rm Comp}(m,\delta)$. Its complement in $\mathbb{S}^{n-1}$, the set of $(m,\delta)$-incompressible vectors is denoted by ${\rm Incomp}(m,\delta)$.
\end{definition}
\begin{definition}[Dominated vectors]
    Let $m \le  n$ and $\alpha<1$. A vector $v \in \mathbb{S}^{n-1}$ is said to be $(m,\alpha)$-dominated if it holds
    \begin{equation}
       \| v_{[m+1:n]}
\| \le \alpha\sqrt{m} \|v_{[m+1:n]}\|_{\infty},
\end{equation}
where $v_{[m+1:n]}$ is the vector in $\mathbb{R}^{n-m}$ obtained removing the entries with the $m$ largest coordinates moduli of $v$ and then putting the remaining ones in non-increasing order. The set of such vectors is denoted by ${\rm Dom}(m,\alpha)$.
\end{definition}
Notice that ${\rm Sparse}(m) \cap \mathbb{S}^{n-1} \subseteq  {\rm Dom}(m,\alpha)$, since for $m$-sparse vectors, $v_{[m+1:n]} = 0$.\\
The following lemma, valid for a general random matrix $A_n$, is used in the proof of Theorem \ref{thm:least-SV-sparse}.
\begin{lemma}[{\cite[Lemma 2.7]{BasakRudelson19}}]
\label{lemma:distance}
    Let $A_n$ be any $n\times n$
random matrix. For $x \in [n]$, let
${A}_{n,x} \in \mathbb{C}^n$ be the $x$-th column of
$A_n$, and let $H_{n,x}$
be the subspace of $\mathbb{C}^n$ spanned by 
$\{{A}_{n,y} : y \in [n]\setminus\{x\}\}$. Then for any $\bar \eps,\rho > 0$, and
$M < n$,
\begin{equation*}
        \p\left(\inf_{z \in {\rm Incomp}(M,\rho)} \|A_nz\| \le \bar \eps \rho^2 \sqrt{\scaling/n^2}\right)
    \le \frac1{M} \sum_{x \in [n]} \p\left({\rm dist} ({A}_{n,x}, H_{n,x})\le \rho \sqrt{\scaling/n}\bar \eps\right).
 \end{equation*}
\end{lemma}
As commented in \cite[p. 2370]{BasakRudelson19} the lemma holds even when we intersect with the event $\{\|A_n-\E[A_n]\|\le K\sqrt{\scaling}\}$, for some $K>0$, which is the form we are going to use.\\

We also need a bound for the contribution to \eqref{eq:sigman} depending on compressible and dominated vectors. In principle given for matrices with zero mean, it can be adapted to the adjacency matrix case via a \textit{folding trick}, which we explain here.
First of all, writing $a_{xy}= {\rm \bf I}_{xy} X_{xy}$ where ${\rm \bf I}_{xy} \sim \ber(2p_{x,y})$ and $X_{xy}\sim \ber(\frac12)$, we have a dilute product structure, but matrix entries are not centered. 
Consider then the rectangular $\lfloor n/2\rfloor \times n$ matrix 
$\hat A_n = A^{(1)}_n - A^{(2)}_n,$
where $A^{(1)}_n$ contains the first $\lfloor n/2 \rfloor$ rows of $A_n$ and $A^{(2)}_n$ contains the next $\lfloor n/2 \rfloor$ rows of $A_n$. The same can be done for a diagonal matrix $D_n$.

\begin{proposition}[{\cite[Proposition 7.3]{BasakRudelson-invertibility}}] \label{thm:least-SV-prop}
Fix $R \ge 1$, $K\ge 1$, and let $D_n$ be a diagonal matrix with complex entries such that
$\|D_n\| \le R\sqrt{\scaling}$.
If $w_x^\pm = 1$ for every $x \in [n]$, Assumption \ref{assumption:bulk-sparse} holds, and $A_n$ has zero diagonal, there
exist positive constants $C,\bar C,\tilde C,c,\bar c$, 
depending only on $K$ and $R$, 
such that for $n/\scaling\le M\le cn$:
\begin{equation}
\begin{split}
    \p \bigg(\exists z \in &{\rm Dom}(M,(C(K+R))^{-4})\cup {\rm Comp}(M,\rho) :\\
    &\|(\hat A_n +\hat D_n)z\|_2 \le \bar C (K+R)\rho \sqrt{\scaling} \ \text{  and  }\ \|\hat A_n\|\le K\sqrt{\scaling}\bigg) \le e^{-\bar c\scaling},
    \end{split}
\end{equation}
where $\rho = (\tilde C(K+R))^{-\ell_0-6}$ and $\ell_0=\lceil \frac{\log(n/8\scaling)}{\log(\sqrt{\scaling)}} \rceil$.
\end{proposition}

The Proposition is stated for the matrix $\hat A_n$, which has centered entries. By the triangle inequality it also holds $\|\hat A_nx\|^2\le 2\|A_nx\|^2$, so a comparable estimate also holds for the original matrix.

As observed in \cite[pp. 477-478]{BasakRudelson-invertibility}, in the homogeneous setting where $p_{x,y}\equiv p$, the proof of Proposition \ref{thm:least-SV-prop} uses that entries $\hat a_{xy}$ of the matrix $\hat A_n$ with $y\neq x, x+\lfloor n/2\rfloor$ (i.e., which do not correspond to diagonal entries of $A_n$), have the product structure $\hat{\rm \bf I}_{xy} \hat X_{xy}$, where $\hat{\rm \bf I}_{xy}$ is Bernoulli of parameter $\bar p (2-\bar p)<1$, where $\bar p =2p$ (we assume $p<1/2$) and $\hat X_{xy}$ satisfies the hypothesis of the Theorem. This is done by defining some auxiliary variables $\hat \theta_{xy}$ and $\hat X_{xy}$ such that
    \begin{equation} \label{eq:theta}
        \p(\hat \theta_{xy}=1)=\p(\hat \theta_{xy}=2)=\frac{1-\bar p}{2-\bar p} ,\qquad \p(\hat \theta_{xy}=3)=\frac{\bar p}{2-\bar p},
    \end{equation}
        \begin{equation} \label{eq:X}
        \hat X_{xy} = X_{xy} \one_{\{\hat \theta_{xy}=1\}}-X_{x+\lfloor {n}/{2}\rfloor y} \one_{\{\hat \theta_{xy}=2\}}+(X_{xy}-X_{x+\lfloor{n}/{2}\rfloor y}) \one_{\{\hat \theta_{xy}=3\}},
    \end{equation}
    and verifying that $\hat a_{xy}$ and $\hat{\rm \bf I}_{xy} \hat X_{xy}$ have the same distribution.  As observed in \cite[p.\ 55]{BasakRudelson19-arxiv} the proposition holds also for complex valued diagonal matrices $D_n$ such as in Theorem \ref{thm:least-SV-sparse}, and this is the way it is used here.
After this discussion, we are ready to explain the proof of Theorem \ref{thm:least-SV-sparse}, as given in \cite{BasakRudelson19}.

\begin{proof}[Sketch of proof of Theorem \ref{thm:least-SV-sparse}]
Let $\Omega = \{\|A_n-\E[A_n]\|\le K\sqrt{\scaling}\}$. This event holds \whp thanks to \cite[Theorem 1.7]{BasakRudelson-invertibility}, however the estimates do not hold w.v.h.p. as in Theorem \ref{thm:LLN-A}. Set
\[
V= \mathbb{S}^{n-1}\setminus \left({\rm Dom}(cn,(C(K+R))^{-4}) \cup {\rm Comp}(cn,\rho) \right),
\]
where the constants come from Proposition \ref{thm:least-SV-prop}. Then, the contribution of vectors in $\mathbb{S}^{n-1}\setminus V$ is controlled, and then one just needs to consider the bound over $V$. However, elements of $V$ have well distributed components, which allows the classical small probability ball bound. We can first use Lemma \ref{lemma:distance} and notice that it suffices to have a bound, for any fixed $x \in [n]$, on the probability of the event 
\begin{equation}
    \left\{ {\rm dist} \left((A_n+D_n)_{x},H_{n,x}\right)\le \rho\sqrt{{\scaling}/{n}} \bar \eps\right\} \cap \Omega,
\end{equation}
where $(A_n+D_n)_{x}$ is the $x$-th column of $A_n+D_n$, and $H_{n,x}$ the space generated by the other columns.

For any $u \in \mathbb{S}^{n-1}\cap H_{n,x}^\perp$, we have $u^t (A_n+D_n)_{x}\le {\rm dist} \left((A_n+D_n)_{x},H_{n,x}\right)$, and we only need to bound
\begin{equation}
\p\left( \left\{
% u \in \mathbb{S}^{n-1} \cap H_{n,x}^\perp : 
\left|\bar u^t (A_n+D_n)_{x}\right|
\le\rho\sqrt{{\scaling}/{n}} \bar \eps \right\} \cap \Omega \right).
\end{equation}
Using again Proposition \ref{thm:least-SV-prop} we can choose $u \in  V$, which implies that, if $J$ is the support of $u_{[cn+1,n]}$,
\begin{equation} \label{eq:norms}
    \|u_J\|_{\infty} \le \frac{C(K+R)^4}{\sqrt{cn}}\|u_J\| \qquad \text{and} \qquad \|u_J\|\ge \rho.
\end{equation}
Moreover, choosing the right phase for $u$, which does not affect orthogonality and the norms in the previous bounds, we can assume that $\sum_{y\in J}{\rm Re}({u}_y)^2 \ge \|u_J\|^2/2$.
Now, since columns are independent, we can condition on $H_{n,x}$ and use the randomness of the $x$-th column only, and $u$ is deterministic. 
 For realizations of the other columns satisfying
\eqref{eq:norms}, we drop $\Omega$ and project onto the real axis. Set
\[
t=\rho\sqrt{\scaling/n}\,\bar\eps,
\qquad r_y=\operatorname{Re}(u_y),
\qquad
Z=\sum_{y\in J}r_y(a_{yx}-p_{yx}).
\]
Since $\operatorname{Re}(\overline{u_y})=r_y$,
conditioning also on the entries of the $x$-th column
outside $J$ makes
$\operatorname{Re}(\bar u^t(A_n+D_n)_x)$ equal to $Z$
plus a deterministic shift. Consequently,
\[
\p\left(
|\bar u^t(A_n+D_n)_x|\le t
\,\middle|\,H_{n,x}
\right)
\le \mathcal F(Z,t),
\]
where $u$ and $J$ are fixed under the conditioning and
$\mathcal F$ denotes the L\'evy concentration function,
\[
\mathcal F(Z,t)
=\sup_{w\in\mathbb R}\p(|Z-w|\le t).
\]
By the choice of phase and $p_{yx}\asymp\scaling/n$,
the variance $\sigma^2$ of $Z$ satisfies
\[
\sigma^2
=\sum_{y\in J}r_y^2p_{yx}(1-p_{yx})
\asymp \frac{\scaling}{n}\sum_{y\in J}r_y^2
\asymp \frac{\scaling}{n}\|u_J\|^2.
\]
Moreover,
\[
\sum_{y\in J}
\E\bigl[|r_y(a_{yx}-p_{yx})|^3\bigr]
\le C\frac{\scaling}{n}\|u_J\|_3^3.
\]
The Berry--Esseen theorem
\cite[Theorem 2.2.17]{Stroock} therefore gives
\begin{align*}
\mathcal F(Z,t)
&\le C\frac{t}{\sigma}
 +C\frac{(\scaling/n)\|u_J\|_3^3}{\sigma^3}\\
&\le C\bar\eps
 +C\frac{\|u_J\|_\infty}{\sqrt{\scaling/n}\,\|u_J\|}
\le C\bar\eps+\frac{C'}{\sqrt{\scaling}},
\end{align*}
uniformly over the realizations under consideration.
Averaging over the other columns and adding the
probability of the exceptional event on which the
normal-vector estimates fail gives the required bound.

\end{proof}

We finally state the analogous of Theorem \ref{thm:least-SV-sparse} in our setting.
\begin{corollary} \label{cor}
    Suppose that Assumptions \ref{ass:rank-one} and \ref{assumption:bulk-sparse} hold. Then the conclusion of Theorem \ref{thm:least-SV-sparse} remains valid for the directed Chung--Lu adjacency matrix $A_n$, with constants depending additionally only on the constants in Assumption \ref{ass:rank-one}.
\end{corollary}
\begin{proof}
The proof is almost the same, we highlight the major changes.
To control the probability of the event $\Omega = \{\|A_n-\E[A_n]\|\le K\sqrt{\scaling}\}$, for large $K$, we can still use the result by \cite{BandeiraVanHandel} used in the proof of \cite[Theorem 1.7]{BasakRudelson-invertibility}.\\
Theorem \ref{thm:least-SV-sparse} is stated for a matrix $ A_n $ with zero diagonal, which is not our case. However, we can simply condition on the diagonal and later take an average over the diagonal entries, as observed in \cite[p.\ 7]{BasakRudelson19-arxiv}. Observe also that the diagonal array $\Lambda_n$ containing diagonal entries has spectral norm at most $O(\sqrt{\scaling})$. This is observed in \cite[p. 472]{BasakRudelson-invertibility}.\\

The inhomogeneity assumption affects the folding argument described in the proof of Proposition \ref{thm:least-SV-prop}, however, under Assumption \ref{assumption:bulk-sparse} it holds $p_{x+\lfloor n/2\rfloor,y}=p_{x,y}$ and $p_{y,x+\lfloor n/2\rfloor}=p_{y,x}$, so that \eqref{eq:theta} and \eqref{eq:X} still hold and folding argument applies to the adjoint.
For the adjoint of a column-deleted matrix, retain
the original row pairs and discard the pair containing
the deleted row, if any. At least $\lfloor n/2\rfloor-1$
independent row pairs remain, so the same argument
gives the compressible/dominated-vector estimate
needed for the normal vectors, with modified constants.
Finally, as observed above, Proposition \ref{thm:least-SV-prop} is stated in \cite{BasakRudelson-invertibility} for real shifts $D_n$ of the matrix $\hat A_n$, but it is observed in \cite[Remark 3.10]{BasakRudelson19} and later in \cite[Proposition 3.4]{BasakRudelson19} that the argument can be adapted to the case where $D_n$ is complex valued.\\

Away from the exceptional diagonal positions, a folded
entry is the difference of two independent Bernoulli
variables with the same parameter $p_{xy}$. It has the
distribution of a product of an independent Rademacher
variable and a Bernoulli variable with parameter
$2p_{xy}(1-p_{xy})\asymp s_n/n$.
Thus the support-pattern and small-ball estimates
in the proof of
\cite[Proposition 7.3]{BasakRudelson-invertibility}
hold uniformly, with modified constants. The subsequent
Chernoff and net arguments apply also to the retained
row pairs of the adjoint minors. The exceptional diagonal
positions are treated as in the cited proof.
% The small-ball and compressible/dominated-vector estimates used in the proof are uniform when the Bernoulli parameters are bounded above and below by fixed multiples of $s_n/n$, as follows from Assumption 2.1.
 Hence the proof of Theorem 4.1 carries over with modified constants.

\end{proof}

\section{Analysis of outlier: rank-one case}
\label{sec:sp-outlier-1}
In this section we study the outlier for rank-one model. In what follows $v^\pm_x=\frac{w^\pm_x}{\sqrt{\w}}$.
\subsection{Preliminary bounds}
\label{sec:sp-lemmata}
Let $D_x^+$ denote the out-degree of a vertex $x \in [n]$. Since $D_x^+$ is a sum of independent Bernoulli random variables, the following concentration inequalities hold (see, e.g., \cite[Prop.\ 2.21]{hofstad2016random}):
\begin{gather} 
    \label{eq:chernoff1}
        \p(D_x^+ \ge \E[D_x^+]+t)\le \exp\left(-\frac{t^2}{2(\E[D_x^+]+t/3)}\right),\\
    \label{eq:chernoff2}
        \p(D_x^+ \le \E[D_x^+]-t)\le \exp\left(-\frac{t^2}{2\E[D_x^+]}\right).
\end{gather}
Choosing $t= \scaling^{\frac 23}$ in \eqref{eq:chernoff1} and \eqref{eq:chernoff2}, and using  Assumption \ref{assumption:xi}, we obtain
\begin{equation}
    \p(\max_{x \in V}|D_x^+-w_x^+\scaling|\ge \scaling^{\frac 23})\le 2n \exp\left(-\scaling^{\frac13}\right)= o\left(e^{-\log(n)^\eta}\right),
\end{equation}
for some $\eta>\frac 4 3$.
Choosing $t=m\E[{D_x^+}]$ in \eqref{eq:chernoff1} and \eqref{eq:chernoff2}, we get
\begin{gather}
    \p\left(\frac1{D_x^+} \ge \frac{1}{\E[{D_x^+}](1-m)}\right)\le \exp\left(-\frac{m^2\E[{D_x^+}]}{2}\right),\\
    \p\left(\frac1{D_x^+} \le \frac{1}{\E[{D_x^+}](1+m)}\right)\le \exp\left(-\frac{m^2 \E[{D_x^+}]}{2(1+m/3)}\right).
\end{gather}
Taking $m=\scaling^{-\frac13}$, since $(1\pm m)^{-1}=1\mp m+o(m)$, the following lemma holds.
\begin{lemma}  \label{lem:concentration}
There exists $\eta>\frac43$ such that,
\begin{equation}
    \p\left(\max_{x \in V}\left|\frac{w_x^+\scaling}{D_x^+}-1\right|\ge 2\scaling^{-\frac13}\right)\le 2n \exp\left(-\frac{c\scaling^{\frac13}}3\right)\le \exp(-\log(n)^{\eta}).
\end{equation}
\end{lemma}

\begin{lemma}  \label{lem:5}
        $| (\y)^t \centered \x | = \Op\left( \sqrt{\frac{\scaling}{n}}\right).$
\end{lemma}

\begin{proof}
It follows by Markov  inequality, observing that
    \begin{equation}
        \E[| (\y)^t \centered \x|] \le  \sqrt{\var((\y)^t \centered \x)}, 
    \end{equation}
    and by direct computation 
\begin{equation}
\operatorname{Var}((v^-)^t C_n v^+)
=\sum_{x,y}\frac{(w_x^-)^2(w_y^+)^2}{\w^2}
 p_{xy}(1-p_{xy})\asymp\frac{s_n}{n}.
\end{equation}
\end{proof}
\begin{lemma} \label{lem:k=1}
    $
(\y)^t \centered \x = \ohp( \scaling)$.
\end{lemma}
\begin{proof}
Recall that $\centered=A_n-\E[A_n]$. Hoeffding  inequality gives,  for every $\eps>0$,
\begin{equation}
\begin{split}
    \p
    \Bigg(\Bigg|\sum_{x,y \in [n]}v_x^- a_{xy} v_y^+-\sum_{x,y \in [n]}&\E\left[v_x^- a_{xy} v_y^+\right]\Bigg|> \eps\scaling\Bigg)\\ &\le 2 \exp \left(-\frac{2\eps^2 \scaling^2}{n^2 (\max_{x,y} w_x^+ w_y^{-}/\w)^2}\right).
    \end{split}
\end{equation}
Since weights are bounded, the \rhs is at most $O(\exp(-2\eps^2(\log(n))^{2\xi}))$.
\end{proof}

\subsection{Perturbation of non-Hermitian matrices}

\label{subsec:sp-perturb}

To obtain a uniform control on the spectrum of a perturbed matrix, we use the following theorem of Bauer--Fike, proved in \cite{BF}. See also \cite[Th.\ 4]{Coste} for a modern formulation. We denote by $\mathcal{B}(\lambda,\eps_n)$ the complex ball of center $\lambda \in \mathbb{C}$ and radius $\eps_n>0$.

\begin{theorem}[Bauer--Fike, \cite{BF}] \label{thm:Bauer-Fike}
    Let $\deterministic$ be an $n \times n$ matrix such that $\deterministic=P_nDP_n^{-1}$ for an invertible matrix $P_n$ and a diagonal matrix $D$. Let $H_n$ be a $n \times n$ arbitrary matrix and $\eps_n = \|P_n\|\|P_n^{-1}\|\|H_n\|$. 
    \begin{enumerate}
        \item[(i)] Then the spectrum of $\deterministic+H_n$ is contained in the union $\bigcup_{i= 1}^n \mathcal{B}(\lambda_i(\deterministic),\eps_n)$.
        \item[(ii)] Moreover, if for $I \subseteq [n]$ it holds
        $$\bigcup_{i \in I} \mathcal{B}(\lambda_i(\deterministic),\eps_n) \cap \bigcup_{i \in I^\complement} \mathcal{B}(\lambda_i(\deterministic),\eps_n)=\emptyset,$$
        then $\deterministic+H_n$ has exactly $|I|$ eigenvalues inside $\bigcup_{i \in I} \mathcal{B}(\lambda_i(\deterministic),\eps_n)$.
    \end{enumerate}
\end{theorem}
In the rank-one case, the previous statement can be specialized as follows.
    \begin{lemma}[{\cite[Lemma A.1]{Coste}}] \label{lem:spectrum}
Let $\xx$ and $\yy$ be two vectors of $\R^n$ and $\deterministic =\xx\yy^t$. Let $H_n$ be a real $n \times n$ matrix.
\begin{enumerate}
\item[(i)] The eigenvalues of the matrix $\deterministic + H_n$ are contained in  
$\mathcal{B}(0, \eps_n) \cup \mathcal{B}(\yy^t\xx, \eps_n)$, where
$$\eps_n = 2\|\xx\|^2\|\yy\|^2(\yy^t\xx)^{-2}\|H_n \|;$$
\item[(ii)] If 
$\mathcal{B}(0, \eps_n) \cap \mathcal{B}(\yy^t\xx, \eps_n) = \emptyset$, then there is exactly one eigenvalue of $\deterministic + H_n$ inside 
$\mathcal{B}(\yy^t\xx, \eps_n)$ and all the other eigenvalues of $\deterministic + H_n$ are contained in $\mathcal{B}(0, \eps_n)$.
\end{enumerate}
\end{lemma}
 
 \subsection{Existence of the outlier}
 \label{subsec:existence-outlier}
 Consider the real matrix $\centered \coloneq A_n-\E[A_n]$, so that $A_n=\E[A_n]+\centered$.
We choose $H_n=\centered$,
$\xx = \sqrt{\scaling} v^+$ and $\yy=\sqrt{\scaling} v^-$. To prove Theorem \ref{thm:LLN-A}, we only need to check that $$\eps_n=2 \|\x\|^2\|\y\|^2((\y)^t\x)^{-2}\|\centered \|= \Ohp(\sqrt{\scaling}).$$
Since $\|\x\|^2 \|\y\|^2 ((\y)^t\x)^{-2}$ is bounded, it suffices to control the spectral norm of $\centered$, as provided by the next proposition.
\begin{proposition}
\label{prop:prop-W}
If Assumptions \ref{ass:rank-one} and \ref{assumption:xi} hold, there exist $K_1>0$ and $\eta>2$ such that, for large $n$,
\begin{equation}
    \p\left(\|\centered\|\ge K_1\sqrt{\scaling}\right)\le e^{-(\log(n))^{\eta}}.
\end{equation}
\end{proposition}
The proof of this proposition relies on the following lemma.
\begin{lemma} \label{lem:lem-W}
If Assumptions \ref{ass:rank-one} and \ref{assumption:xi} hold, there exists a constant $K_2>0$ such that, if  $1\ll m \ll \scaling^{1/4}$, 
    \begin{equation} \label{eq:claim-W}
    \E[\|\centered\|^{2m}]\le6n \left({K_2 \scaling}\right)^m.
\end{equation}
\end{lemma}
\begin{proof}[Proof of Proposition \ref{prop:prop-W}]
Taking $$t={\left({K_2 \scaling}\right)^{1/2}}+\scaling^{1/4}\log(n)^{\xi/4},$$ by the $2m$--moment Markov  inequality we obtain
\begin{equation} 
\begin{split}
    \p(\|\centered\|\ge t ) &\le \frac1 {t^{2m}}{\E[\|\centered\|^{2m}]} \le 6n \left(\frac{\left({K_2 \scaling}\right)^{1/2}}{{\left({K_2 \scaling}\right)^{1/2}}+\scaling^{1/4}\log(n)^{\xi/4}}\right)^{2m} \\
    &\le 6n \left(1-\frac{\scaling^{1/4}\log(n)^{\xi/4}}{{\left({K_2 \scaling}\right)^{1/2}}+\scaling^{1/4}\log(n)^{\xi/4}}\right)^{2m} 
    \le  6n \exp\left(- c m \frac{\log(n)^{\xi/4}}{\scaling^{1/4}}\right),
    \end{split}
\end{equation}
for some constant $c>0$. Choosing $m=\lfloor\scaling^{1/4}/(\log(n))^{\delta}\rfloor\ll \scaling^{1/4}$, where $0<\delta<\frac{\xi}{4}-2$, gives the claim for some $K_1>0$, $\eta>2$.
\end{proof}
We first give a proof of Lemma \ref{lem:lem-W} which relies on the estimated derived by \cite{BandeiraVanHandel}. 
We then provide a second combinatorial proof, which we believe it is interesting on its own. This generalizes the approach by \cite{FK} to non-symmetric matrices, and identifies the relevant contribution to the upper bound.

% \begin{lemma}\label{lem:lem-W}
% There is $K_2>0$ such that, for
% $1\ll m\ll s_n^{1/4}$,
% \[
% \mathbb E\|C_n\|^{2m}\le6n(K_2s_n)^m.
% \]
% \end{lemma}
\begin{proof}[Proof 1 of Lemma~\ref{lem:lem-W}]
Consider the symmetric dilation
\[
X_n=\begin{pmatrix}0&C_n\\C_n^t&0\end{pmatrix}.
\]
Its upper-triangular entries are independent and
centered, bounded by one, and its maximum row sum
of entry variances is at most $K s_n$.
Moreover $\|X_n\|=\|C_n\|$.
Corollary~3.12 and Remark~3.13 of
\cite{BandeiraVanHandel} therefore give constants
$K_0,c_0>0$ such that
 \[
\mathbb P(\|C_n\|\ge K_0\sqrt{s_n}+t)
\le 2n\exp(-c_0t^2),\qquad t\ge0.
\]
The constants in this estimate are uniform in the
Bernoulli parameters: the bound uses the variance
sums and the uniform bound $|C_{xy}|\le1$, and requires
no positive lower bound on the individual variances.
Taking $t=\sqrt{s_n}$ proves
Proposition~\ref{prop:prop-W}, since
$s_n\gg(\log n)^\xi$ and $\xi>8$.
For the moment bound, let
$Y=(\|C_n\|-K_0\sqrt{s_n})_+$.
Integrating the tail yields
\[
\mathbb E Y^{2m}\le2n c_0^{-m}m!.
\]
Using $(x+y)^{2m}\le2^{2m-1}(x^{2m}+y^{2m})$, $m!\le m^m$, and then $m\ll s_n^{1/4}$,
gives the claimed inequality after increasing $K_2$.
This proof also applies to the higher-rank model:
only independence, centering, bounded entries and
$p_{xy}\le K s_n/n$ are needed.
\end{proof}

% We are left with the proof of Lemma \ref{lem:lem-W}. 
\begin{proof}[Proof 2 of Lemma \ref{lem:lem-W}]
To embark on the proof, we make some notational preliminaries. We write
\begin{equation} \label{eq:trace}
    \begin{split}
\|\centered\|^{2m}
&=\lambda_1(\centered\centeredstar)^{m}
= \lambda_1((\centered\centeredstar)^{m})
\le \trace((\centered\centeredstar)^m)\\
&= \sum_{x_1,x_2,\dots,x_{2m}} \centeredentry_{x_1 x_2}\phantom{^*} \centeredentry^*_{x_2 x_3} \centeredentry_{x_3 x_4}\phantom{^*} \centeredentry^*_{x_4 x_5} \cdots \centeredentry_{x_{2m-1} x_{2m}}\phantom{^*} \centeredentry^*_{x_{2m} x_{1}}\\
&= \sum_{x_1,x_2,\dots,x_{2m}} \centeredentry_{x_1 x_2} \centeredentry_{x_3 x_2} \centeredentry_{x_3 x_4} \centeredentry_{x_5 x_4} \cdots \centeredentry_{x_{2m-1} x_{2m}} \centeredentry_{x_{1} x_{2m}},
%= \sum_{\pat \in \mathcal{P}_m} \prod_{s=1}^{2m} W_{e_s},
\end{split}\end{equation}
where the indices $x_1,\dots,x_{2m}$ run from $1$ to $n$.
Our aim is to provide a bound on the expectation of this sum.
Writing $e^-=x$, $e^+=y$ for any $e=(x,y) \in E$, we can define
\begin{equation}
    \mathcal{P}_m\coloneq \left\{ \pat=(e_1,\dots,e_{2m}) \in E^{2m} : e_{2i-1}^+=e_{2i}^+ ,\ \ e_{2i}^-=e_{2i+1}^{-} \text{ for all } i \in \{1,\dots,m\} \right\},
\end{equation}
with the convention that $e_{2m+1}=e_1$. This set denotes a family of alternating edge-paths. 
For instance, if $m=1$, elements of $\mathcal{P}_m$ will be of the form $((x_1,x_2),(x_1,x_2))$ for $x_1,x_2 \in [n]$, while for $m=2$ we will have elements of the form $((x_1,x_2),(x_3,x_2),(x_3,x_4),(x_1,x_4))$ for $x_1,x_2,x_3,x_4 \in [n]$, and so on.

Then \eqref{eq:trace} reads
\begin{equation}
    \|\centered\|^{2m}\le \sum_{\pat \in \mathcal{P}_m} \prod_{s=1}^{2m} \centeredentry_{e_s(\pat)},
\end{equation}
where $e_s(\pat)$ denotes the $s$-th edge of $\pat$ and, if $e=(x,y)$, we set $\centeredentry_{e}=\centeredentry_{xy}$. Since $\centered=A_n-\E[A_n]$ has independent and centered entries, when taking expectation on both sides, we can restrict the sum to the smaller set of paths having each edge repeated at least twice.
Denote this subset of $\mathcal P_m$ by $\mathcal{R}_m$. Also denote by ${\ell(\pat)}$ the number of distinct edges in $\pat \in \mathcal{R}_m$ and by $E(\pat)=(\tilde{e}_1,\dots,\tilde{e}_{{\ell(\pat)}})$ the ordered sequence of such distinct edges in $\pat \in \mathcal{R}_m$. We get
\begin{equation} \label{eq:trace2}
    \E[\|\centered\|^{2m}] \le\sum_{\pat \in \mathcal{R}_m} \prod_{s=1}^{{\ell(\pat)}} \E\left[\centeredentry_{\tilde{e}_s(\pat)}^{k_s(\pat)}\right],
\end{equation}
where $k_s(\pat)\ge 2$ denotes the multiplicity of $\tilde e_s(\pat)$ in $\pat$. \\
% We are ready to prove Lemma \ref{lem:lem-W}.

Given a path $\pat \in \mathcal{R}_m$, consider the sequence of vertices defined by the following iterative procedure. 
For $j=0$, set $\tilde{v}_0=\tilde{e}_1^-$. Then, for $j=1,\dots,{\ell(\pat)}$: 
\begin{itemize}
\item set $\tilde{v}_j=\tilde{e}_j^+$ if the first occurrence of $\tilde{e}_j$ in $\pat$ occupies an odd position;
\item set $\tilde{v}_j=\tilde{e}_j^-$ otherwise (if the first occurrence of $\tilde{e}_j$ in $\pat$ occupies an even position).
\end{itemize}
Let $V(\pat)=(\tilde{v}_0,\dots,\tilde{v}_{\ell(\pat)})$. 
Notice that while $E(\pat)$ has exactly ${\ell(\pat)}$ distinct edges, $V(\pat)$ has exactly ${\ell(\pat)}+1$ vertices and maybe some of them will be repeated. Let $\#V(\pat)$ denote the number of distinct vertices in $V(\pat)$. \\

In what follows, we want to identify a subfamily of paths in $\mathcal{R}_m$ that provides the main contribution to the sum in \eqref{eq:trace2}. To this aim, for $1\le p-1\le l\le m $, let us define
\begin{equation}
    \mathcal{R}_{m,l,p}\coloneq \{\pat \in \mathcal{R}_m \,\, | \,\, {\ell(\pat)}= l, \#V(\pat)=p\},
\end{equation}
so that \eqref{eq:trace2} becomes
\begin{equation} \label{eq:trace3}
    \E[\|\centered\|^{2m}] \le\sum_{l=1}^m \sum_{p=1}^{l+1}\sum_{\pat \in \mathcal{R}_{m,l,p}} \prod_{s=1}^{{\ell(\pat)}} \E\left[\centeredentry_{\tilde{e}_s(\pat)}^{k_s(\pat)}\right].
\end{equation}
We will show that the sum over $\mathcal{R}_{m,m,m+1}$ will give the leading order for the total sum.\\

We use the following bound on the total number of paths:
\begin{equation}\label{eq:path-count}
    |\mathcal R_{m,l,p}|
    \le n^p \binom{2m}{2p-2} C_{p-1}
             p^{4(m-p+1)}
    \le 16^m m^{4(m-l)} n^p m^{4(l+1-p)}.
\end{equation}
Indeed, each alternating path $\pat$ determines a closed
vertex walk $(x_1,\ldots,x_{2m},x_1)$ visiting exactly
$p$ vertices. This correspondence is injective, since
the directed edges are recovered from the vertex walk by
\[
e_j=
\begin{cases}
(x_j,x_{j+1}),& j \text{ odd},\\
(x_{j+1},x_j),& j \text{ even},
\end{cases}
\qquad x_{2m+1}=x_1.
\]
Every undirected edge of this walk occurs at least twice,
because every directed edge of $\pat$ does.
The walk-counting bound in \cite[Section 3.2]{FK},
which also allows loops, therefore gives the first
inequality. Although forgetting orientations may reduce
the number of distinct edges, this bound depends only
on the length $2m$ and the number $p$ of vertices.
The second inequality follows from
$\binom{2m}{2p-2}\le 4^m$ and $C_{p-1}\le 4^m$,
using $p\le m$ unless $p=m+1$, when the exponent
$4(m-p+1)$ is zero.

At this point, let us observe that, for every $k\ge 2$ and $x,y \in [n]$, it holds
\begin{equation}
    \left|\E[\centeredentry_{xy}^k]\right|=(1-p_{x,y})^kp_{x,y}+(-p_{x,y})^k (1-p_{x,y}) \le p_{x,y}.
\end{equation}
Indeed, this is immediate for $k$ even, while for $k$ odd it follows from $(1-p_{x,y})^{k}-p_{x,y}^{k-1}(1-p_{x,y})<1$.
Consequently, the absolute contribution of each path with $l$ distinct edges is bounded by $p_{\max}^l$, where $p_{\max}=\max_{x,y \in [n]} p_{x,y}$.

Combining \eqref{eq:path-count} with the bound
$p_{\max}^l$ for the absolute contribution of each path,
we obtain
\begin{equation}\label{eq:trace-4}
\begin{split}
\E[\|\centered\|^{2m}]
&\le 16^m \sum_{l=1}^m m^{4(m-l)}p_{\max}^l
       \sum_{p=1}^{l+1} n^p m^{4(l+1-p)}\\
&=16^m \sum_{l=1}^m m^{4(m-l)}
       n^{l+1}p_{\max}^l
       \sum_{d=0}^{l}\left(\frac{m^4}{n}\right)^d\\
&\le 2\cdot16^m \sum_{l=1}^m E_{m,l},
\end{split}
\end{equation}
where $E_{m,l}=m^{4(m-l)}n^{l+1}p_{\max}^l$.
The last inequality holds for all sufficiently large $n$,
since $m^4=o(\scaling)$ and $\scaling=o(n)$.

Consider
\begin{equation}
    {\frac{E_{m,m}}{E_{m,l}}}= {\frac{n^{m+1}p_{\max}^m}{m^{4(m-{\ell(\pat)})} n^{l+1}p_{\max}^l}}=\left(\frac{n p_{\max}}{m^4}\right)^{m-l}.
\end{equation}
Since $p_{\max}\asymp \scaling/n$ by \eqref{eq:sp-chung-lu} and the boundedness of weights, and since $1\ll m\ll \scaling^{1/4}$, the term in brackets diverges and is at least $2^{m-l}$ as $n$ grows.
 Hence, for all sufficiently large $n$,
\[
\E[\|\centered\|^{2m}]
\le 2\cdot16^m\sum_{l=1}^m E_{m,l}
\le 4\cdot16^m E_{m,m}
=4n(16np_{\max})^m
\le 6n(K_2\scaling)^m,
\]
after choosing $K_2$ sufficiently large.
This concludes the proof.
\end{proof}
\subsection{Random walk}

\label{subsec:sp-transition-matrix}

We adapt here the proof of Theorem \ref{thm:LLN-A} to the transition matrix case.
\begin{proof}[Proof of Theorem \ref{thm:LLN-T}]
    
Consider the transition matrix of the simple random walk on the directed Chung--Lu graph, $T_n=D_n^{-1}A_n$. We write $\tildecentered=D_n^{-1}(A_n-\E[A_n])$, which is not centered, and then apply Lemma \ref{lem:spectrum} with $H_n=\tildecentered$ and the following choice of vectors
    \begin{equation} \label{eq:vectors}
    \xx=\frac{{\scaling}}{\sqrt{\w}} (D_1^{-1}w_1^+, \dots ,D_n^{-1}w_n^+ ) 
    \quad \text{ and } \quad 
    \yy=v^-=\frac{1}{\sqrt{\w}}(w_1^-, \dots, w_n^-). 
\end{equation}
To conclude the proof, we need to provide a bound for $\tilde \eps_n = 2\|\xx\|^2\|\yy\|^2(\yy^t\xx)^{-2}\|\tildecentered \|.$
We first bound $\|\tildecentered \|$. Since $\tildecentered$ is not centered, we cannot directly apply the machinery developed in Subsection \ref{subsec:existence-outlier}. We then define $$\overlinecentered\coloneq (\E[D_n])^{-1}\left(A_n-  \E[A_n]\right)=(\E[D_n])^{-1}\centered.$$
This matrix is centered and by sub-multiplicativity, it holds
$$\|\tildecentered -\overlinecentered\|=\|(D_n^{-1}-(\E[D_n])^{-1})\centered\|\le \|D_n^{-1}-(\E[D_n])^{-1}\|\|\centered\|.$$
Thanks to the above analysis (Proposition \ref{prop:prop-W}) it holds $\|\centered\|=\Ohp(\sqrt{\scaling})$.\\
By Lemma \ref{lem:concentration}, there exists $\eta>4/3$ such that
\begin{equation}  \label{eq:concentration}
    \p\left(\max_{x \in V}\left|\frac{1}{D_x^+}-\frac1{w_x^+\scaling}\right|\ge 2\scaling^{-\frac43}\right)
    % \le 2n \exp\left(-\frac{c\scaling^{\frac13}}3\right)
    \le \exp(-\log(n)^{\eta}).
\end{equation}
Recalling that $\E[D_x^+]=w_x^+ \scaling$, this implies that $\|D_n^{-1}-(\E[D_n])^{-1}\|=\Ohp(\scaling^{-4/3})$.
Therefore
$$\|\tildecentered -\overlinecentered\|=\Ohp(\scaling^{{1}/{2}-{4}/{3}})=\Ohp(\scaling^{-5/6})=\Ohp(\scaling^{-1/2}).$$ 
We can then repeat the procedure of Subsection \ref{subsec:existence-outlier} for the centered matrix $\overlinecentered$ and obtain 
$$\|\tildecentered\|\le \|\tildecentered -\overlinecentered\|+\|\overlinecentered\|=\Ohp(\scaling^{-1/2}).$$
It remains to bound the other terms appearing in the definition of $\tilde \eps_n$.
Notice that $\xx$ is a random vector and hence the same holds for the unique non-zero eigenvalue of $S=\yy\xx^t$,
which is $$\lambda_1(S)=\yy^t\xx=\sum_{x \in V}\frac{w_x^-}{\w}\frac{w_x^+\scaling}{D_x^+}.$$ 
However, by convexity,
\begin{equation}
    |\lambda_1(S)-1|=\left|\sum_{x \in V}\frac{w_x^-}{\w}\frac{w_x^+\scaling}{D_x^+}-1\right|
    \le\sum_{x \in V}\frac{w_x^-}{\w}\left|\frac{w_x^+\scaling}{D_x^+}-1\right|
    \le \max_{x \in V}\left|\frac{w_x^+\scaling}{D_x^+}-1\right|,
\end{equation}
and the last term is $\Ohp(\scaling^{-\frac13})$ thanks to Lemma \ref{lem:concentration}. Then
\begin{equation}\label{eq:bound}
    |\lambda_1(S)^2-1|\le |\lambda_1(S)-1|\cdot|\lambda_1(S)+1|\le |\lambda_1(S)-1|\cdot(2+|\lambda_1(S)-1|),
\end{equation}
and we conclude that $|\lambda_1(S)^2-1|=\Ohp(\scaling^{-\frac13})$.
Moreover
\begin{equation}
\begin{split}
    \|\xx\|^2=\sum_{x \in V} \frac{1}{\w}\left(\frac{w_x^+\scaling}{D_x^+}\right)^2
     \le \frac{n}{\w}\max_{x \in V}\left(\frac{w_x^+\scaling}{D_x^+}\right)^2
     \le \frac{n}{\w}\left(1+\max_{x \in V}\left|\frac{w_x^+\scaling}{D_x^+}-1\right|\right)^2,
\end{split}
\end{equation}
which yields $\|\xx\|^2=\Ohp(1)$, again by Lemma \ref{lem:concentration}.
Then, it holds $ \tilde \eps_n =\Ohp(\scaling^{-\frac12})$.
Thus, \wetahp, it holds $\mathcal{B}(0, \tilde \eps_n) \cap \mathcal{B}(\lambda_1(S), \tilde \eps_n) = \emptyset$, and, applying Lemma \ref{lem:spectrum}(ii), there exists a unique eigenvalue of $T_n$ around $1$, which is $1$ itself; all the other eigenvalues are contained in $\mathcal{B}(0,\tilde  \eps_n)$. This completes the proof of Theorem \ref{thm:LLN-T}.
\end{proof}

\subsection{Fluctuations around the mean}

\label{subsec:sp-fluctuations}

Later we will make use of the following lemmata.

\begin{lemma} \label{lem:2}
    Let $L=\lceil\log(n)\rceil$. There exists $\eta > 1$ such that
\begin{equation}
    \max_{2\le k \le L} \p\left(
    \left| (\y)^t \centered^k \x - \E [(\y)^t\centered^k\x]\right| >
    \scaling^{k/2} n^{-1/2} \log(n)^{k\xi/4}
    \right)
    = O\left(e^{-(\log(n))^{\eta}}\right).
\end{equation}
\end{lemma}

\begin{proof}
    The proof follows from the moment estimate
    \begin{equation}
       \E[\left|(\y)^t \left(\centered^k - \E [\centered^k]\right)\x\right|^p]< n^{-p/2}(K_3kp)^{kp} \scaling^{\frac{kp}{2}},
    \end{equation}
    where 
    $p\coloneq\frac{\log(n)^\eta}{K_3 k}$.
    This high-moment estimate is obtained adapting \cite[Lemma 6.5]{EKetal} to the inhomogeneous setting as in \cite[Lemma 4.3]{CCH} and observing that in our non--reversible setting the entries of the matrix are already independent, so that there is no need to decompose $\centered$ into the sum of an upper and a lower triangular matrix.
\end{proof}

\begin{lemma} \label{lem:1}
Let $L=\lceil\log(n)\rceil$. There exists a constant $K_1<+\infty$ such that, for $2\le k \le L$,
\begin{equation}
    \left|\E \left[(\y)^t\centered^k\x\right]\right| \le \left(K_1 \scaling\right)^{k/2}.
\end{equation}
\end{lemma}

\begin{proof}
By Proposition \ref{prop:prop-W}, there exist $K_1>0$ and $\eta>2$ such that for the event $\mathcal{A}\coloneq \{\|\centered\|\le K_1\sqrt{\scaling}\}$ it holds $\p(\mathcal{A}^\complement)\le e^{-(\log(n))^{\eta}}$. Then
\begin{equation}
    \left|\E\left[(\y)^t \centered^k \x \right]\right|\le     
    \left|\E\left[(\y)^t \centered^k \x \one_{\mathcal{A}} \right]\right|
    + \left|\E\left[(\y)^t \centered^k \x \one_{\mathcal{A}^\complement} \right]\right|.
\end{equation}
The first term is bounded by  $\|v^-\|\|v^+\|\E[\|\centered\|^k\one_{\mathcal A}] \le K_1\scaling^{{k}/2}$. For the second term, note that $\left((\y)^t \centered^k \x\right)^2$ is bounded by a power of $n$, say $n^{kC'}$, for $C'>0$. Hence by Cauchy--Schwarz, since $k\le \log(n)$ and $\eta>2$
\begin{align*}
    \left|\E\left[(\y)^t \centered^k \x \one_{\mathcal{A}^\complement} \right]\right|
    &  \le \left(\E\left[\left((\y)^t \centered^k \x\right)^2\right]\right)^{\frac12}\p\left({\mathcal{A}^\complement}\right)^{\frac12} \\
    & \le n^{kC'/2}e^{-(\log(n))^{\eta}/2}=\exp\left(\frac{kC'\log(n)-(\log(n))^{\eta}}2\right)=o(1).
\end{align*}
\end{proof}

For notational convenience let $\lambda=\lambda_1(A_n)$ and let $v$ denote a corresponding unit eigenvector.
It holds $ A_nv = \centered v+\E[A_n]v=\lambda v,$
and pre-multiplying by $v^t$,
 $$ v^t \centered v+v^t\E[A_n]v=\lambda.$$
By Theorem \ref{thm:LLN-A}, $\lambda$ is of order $\scaling$ w.v.h.p.\ and, due to Proposition \ref{prop:prop-W}, $v^t \centered v$ has lower order (it holds $\|\centered\|=\Ohp(\sqrt{\scaling})$). We get that $v^t \E[A_n] v /\scaling = v^t (v^+) (v^-)^t v$  does not vanish w.v.h.p., and so does $(v^-)^t v $. 
Moreover, there exists $\eta>1$ and $K>0$, such that the event $\{\|\centered\|\ge K \sqrt{\scaling}\}$  has probability at most $\exp(-\log(n)^\eta)$, and thus the matrix $I_n-\frac{\centered}{\lambda}$ is \wehp invertible, 
so that \wehp the following display holds:
\vspace{0.3cm}
\begin{equation} \label{eq:implications}
\begin{split}
% & \quad Av = \centered v+\E[A_n]v=\lambda v\\
% \Longleftrightarrow 
 \quad (\lambda I_n-\centered)v = \E[A_n] v
\quad \Longrightarrow & \quad \lambda v = \left(I_n-\frac{\centered}{\lambda}\right)^{-1}\E[A_n] v
% \\ \Longrightarrow & \quad   \lambda v
= \sum_{k=0}^{+\infty} \left(\frac{\centered}{\lambda}\right)^k\E[A_n]v \\
% \Longrightarrow & \quad   \lambda (v^-)^t v = \sum_{k=0}^{+\infty} (v^-)^t\left(\frac{\centered}{\lambda}\right)^k\E[A_n]v\\
\Longleftrightarrow & \quad   \lambda (v^-)^t v = \sum_{k=0}^{+\infty} (v^-)^t\left(\frac{\centered}{\lambda}\right)^k\scaling v^+ (v^-)^tv.
\end{split}
\end{equation}
Since $(v^-)^t v\neq0$, \wehp we end up with
% \begin{equation}
    $\lambda = \scaling\sum_{k=0}^{+\infty} (v^-)^t\left(\frac{\centered}{\lambda}\right)^k v^+.$\\
% \end{equation}
We have that, \wehp
\begin{equation} \label{eq:lambda-decomp}
    \lambda 
    % = \scaling(v^-)^t\sum_{k=0}^{+\infty} \left(\frac{\centered}{\lambda}\right)^k v^+
    =\scaling(v^-)^tv^+ +\scaling(v^-)^t \frac{\centered}{\lambda} v^++ \term^{(1)} +\term^{(2)} +\term^{(3)} ,\end{equation}
    where, for $L=\lceil\log(n)\rceil$, 
\begin{alignat}{2}
% \begin{equation}
% \begin{split}
    \term^{(1)}=& \ \scaling \sum_{k=L+1}^{+\infty} (v^-)^t\left(\frac{\centered}{\lambda}\right)^k v^+, \qquad &&\text{(high exponent)} \\
    \term^{(2)}=& \ \scaling \sum_{k=2}^{L} (v^-)^t \frac{\centered^k-\E[\centered^k]}{\lambda^k} v^+, \qquad &&\text{(centering)}\\
    \term^{(3)}=& \ \scaling \sum_{k=2}^{L} (v^-)^t \frac{\E[\centered^k]}{\lambda^k} v^+. \qquad  &&\text{(main contribution)}   
% \end{split}
% \end{equation}
\end{alignat}
We will show that the only relevant contribution is the third one, while the first two are negligible.
For $\term^{(1)}$, Proposition \ref{prop:prop-W} and Theorem \ref{thm:LLN-A} imply that there exist constants $K,K_1>0$ such that \wehp it holds
\begin{equation} \label{eq:R1}
    |\term^{(1)}|
    \le \sum_{k=\log(n)+1}^{+\infty} \scaling\frac{\|v^-\|\|v^+\|\|C_n\|^k}{\lambda^k}
    \le \sum_{k=\log(n)+1}^{+\infty} \left(\frac{K_1\scaling^{1/2}}{K\scaling }\right)^k
    = O\left({\scaling^{-\log(n)/3}}\right).
\end{equation}
To estimate $\term^{(2)}$, applying Lemma \ref{lem:2} and Theorem \ref{thm:LLN-A}, \wethp
\begin{equation} \label{eq:R2}
\begin{split}
    |\term^{(2)}|&\le \scaling \sum_{k=2}^L \frac{\left(\scaling^{1/2} \log(n)^{\xi/4}\right)^k}{n^{1/2}} \left(\frac{1}{K\scaling}\right)^k \\
    &= \frac{\scaling}{n^{1/2}}\sum_{k=2}^L \left(\frac{\log(n)^{\xi/4}}{K\scaling^{1/2}}\right)^{k}
    = O\left(\sqrt{\frac{\log(n)^{\xi}}{n}}\right).
    \end{split}
\end{equation}
Combining the above two estimates, we have
\begin{equation}\label{eq:lambda}
\lambda=s_n(v^-)^tv^+
 +\frac{s_n}{\lambda}(v^-)^tC_nv^+
 +s_n\sum_{k=2}^L\frac{(v^-)^t\mathbb E[C_n^k]v^+}
 {\lambda^k}
 +\Ohp\left(\frac{\log^{\xi/2}n}{\sqrt n}\right).
\end{equation}
It would be tempting to think that $\term^{(3)}$ behaves as $\term^{(1)}$ and $\term^{(2)}$, but it turns out that the term $\scaling(v^-)^tv^+ $ alone does not provide an estimate of $\E[\lambda]$, which has to be given in terms of the entire sum over $k \in \{0,\dots,L\}$ as Lemma \ref{lem:fixed-point} and Lemma \ref{lem:R} will show.
To get this, consider the fixed point equation
\begin{equation} \label{eq:fixed-point}
    x = h(x)\coloneq\scaling \sum_{k=0}^{L}  \frac{(v^-)^t {\E[C^k_n]} v^+}{x^k}
    = \scaling(v^-)^t  v^+ +  \scaling\sum_{k=2}^{L}  \frac{(v^-)^t {\E[C^k_n]} v^+}{x^k}.
\end{equation}
For fixed $n$, $h:(0,+\infty)\to (0\, , +\infty)$ is decreasing as $x$ grows, so that there exists a unique solution $\tilde \lambda$. Moreover choosing $x= y \scaling$, for $y \in (0,\, +\infty)$, and using Lemma \ref{lem:1}, we get that
\begin{equation}
    \frac{h(y\scaling)}{\scaling} =(\y)^t\x+  \sum_{k=2}^{L}  \frac{(v^-)^t {\E[C^k_n]} v^+}{(y\scaling)^k} =(\y)^t\x (1+o(1)),
\end{equation}
so that we conclude that $\tilde{\lambda}=(\y)^t\x\scaling(1+o(1))$.

\begin{lemma} \label{lem:fixed-point}
   In the present setting, it holds
     % \begin{equation} \label{eq:lambda-3}
        $\lambda -\tilde{\lambda}=\scaling\frac{(v^-)^t \centered v^+}{\tilde\lambda}+\ohp\left(\sqrt{\frac{\scaling}{n}}\right).$
    % \end{equation}
\end{lemma}

\begin{proof}
Combining \eqref{eq:lambda} with the definition of $\tilde \lambda$ we get
\begin{equation} \label{eq:equation-H^{(3)}_n}
    \lambda-\tilde\lambda
    = \scaling\frac{(v^-)^t \centered v^+}{\lambda}+\scaling\sum_{k=0}^{L}\left(\frac{1}{\phantom{\tilde\lambda} \!\!\!\!\ \lambda^k}-\frac{1}{\tilde\lambda^k}\right) (v^-)^t {\E[C^k_n]} v^+ + \Ohp\left(\sqrt{\frac{\log(n)^{\xi}}{n}}\right).
\end{equation}
By Theorem \ref{thm:LLN-A} it holds
\begin{equation} \label{eq:frac1}
    \frac{1}{\phantom{\tilde\lambda} \!\!\!\!\ \lambda^k}-\frac{1}{\tilde\lambda^k}= (\tilde\lambda-\lambda)\left(\frac{\sum_{j=0}^{k-1}\tilde\lambda^j \lambda^{k-j-1}}{\lambda^k\tilde\lambda^k}\right)=(\tilde\lambda-\lambda)\Ohp\left(\frac{k}{\scaling^{k+1}}\right),
\end{equation}
so that, applying Lemma \ref{lem:1}, it results
\begin{equation} \label{eq:frac2}
\begin{split}
    \left|\scaling\sum_{k=0}^{\log(n)}\left(\frac{1}{\phantom{\tilde\lambda} \!\!\!\!\ \lambda^k}-\frac{1}{\tilde\lambda^k}\right) (v^-)^t {\E[C^k_n]} v^+\right|
    & \le |\lambda-\tilde\lambda|\, \Ohp\left( \scaling \sum_{k=2}^{L} \frac{k}{\scaling^{k+1}} (K_1 \scaling)^{k/2}\right)\\
    & =\Ohp\left(\frac{|\lambda-\tilde\lambda|}{\scaling}\right),
\end{split}
\end{equation}
where we used that $L=\lceil \log(n)\rceil$ and $\sum_{k=2}^{\log(n)} k/\scaling^{k/2}= O(1/\scaling)$.
As a consequence,
\begin{equation} \label{eq:sp-display}
    \lambda-\tilde\lambda=\scaling\frac{(v^-)^t \centered v^+}{\lambda} + \Ohp\left(\frac{|\lambda-\tilde\lambda|}{\scaling}\right)+ \Ohp\left(\sqrt{\frac{\log(n)^{\xi}}{n}}\right).
\end{equation}
By Lemma \ref{lem:5} and Theorem \ref{thm:LLN-A} it holds $\scaling\frac{(v^-)^t {\centered} v^+}{\lambda} = \Op(\sqrt{\frac{\scaling}{n}})$. Then being $\scaling \gg \log(n)^\xi$ and the second addend in the \rhs smaller, Eq.\ \eqref{eq:sp-display} implies 
\begin{equation} \label{eq:lambdas}
|\lambda-\tilde\lambda|=  \Op\left(\sqrt{\frac{\scaling}{n}}\right).\end{equation}
Consequently, the second addend in the r.h.s.\ of Eq.\ \eqref{eq:sp-display} is absorbed in the other error and
    \begin{equation} \label{eq:lambda-2}
        \lambda -\tilde{\lambda}=\scaling\frac{(v^-)^t \centered v^+}{\lambda}+\ohp\left(\sqrt{\frac{\scaling}{n}}\right).
    \end{equation}
 Reasoning as in \eqref{eq:frac1}, by Lemma \ref{lem:5} and \eqref{eq:lambdas}, we also have 
 \begin{equation} \label{eq:49}
     \left|\left(\frac{1}{\phantom{\tilde\lambda} \!\!\!\!\!\ \lambda}-\frac{1}{\tilde\lambda}\right)(v^-)^t \centered v^+\right| = \ohp(|\lambda-\tilde\lambda|)=\ohp\left(\sqrt{\frac{\scaling}{n}}\right).
 \end{equation} 
Then, we can change the $\lambda$ in the denominator of \eqref{eq:lambda-2} to $\tilde \lambda$. This proves Lemma \ref{lem:fixed-point}.
\end{proof}

\begin{lemma} \label{lem:R}
    It holds $\E[\lambda]-\tilde\lambda = o\left(\sqrt{\frac{\scaling}{n}}\right)$.
\end{lemma}
\begin{proof}
Let $R \coloneq \lambda- \tilde\lambda - \scaling (v^-)^t {\tilde\lambda}^{-1}\centered v^+$. 
By equation \eqref{eq:lambda-2} there exists $\eta>1$ such that, for any $\delta>0$, it holds 
\begin{equation}
    \E[|R|] < \delta \sqrt{\frac{\scaling}{n}}+\left(\E\!\left[\left(\lambda-\tilde\lambda-\scaling\frac{v^- \centered v^+}{\tilde\lambda}\right)^2\right]\right)^{\frac12} \!\!\!\!
    \exp{\left({-\frac{(\log(n))^{\eta}}{2}}\right)}=o\left(\sqrt{\frac{\scaling}{n}}\right).
\end{equation}
Since $|\E[\lambda]-\tilde\lambda|=|\E[R]|\le\E[|R|]$, we conclude.
\end{proof}
We are ready to prove the main theorem.
\begin{proof}[Proof of Theorem \ref{thm:CLT-A}]
   Thanks to Lemma \ref{lem:fixed-point} and Lemma \ref{lem:R}, it holds
   \begin{equation}
    \sqrt{\frac{n}{\scaling}}\left(\lambda_1(A_n)- \E[\lambda_1(A_n)] \right)= 
    \sqrt{n \scaling}\frac{(v^-)^t \centered v^+}{\tilde \lambda} + \ohp(1).
   \end{equation}
The first term of the \rhs is a sum of independent random variables satisfying the hypotheses of Lindeberg  CLT. To identify the variance we just need to recall $p_{x,y}=\scaling v_x^+ v_y^-$ and compute
\begin{equation}
\label{eq:computations}
\begin{split}
    \var\Big(  \sqrt{n \scaling}&\frac{(v^-)^t \centered v^+}{\tilde\lambda}\Big)
    \sim \frac{n\scaling}{(\sum_{x \in [n]} v_x^+ v_x^- \scaling)^2}\sum_{x,y \in [n]} (v_x^-)^2\scaling v_x^+ v_y^-(1-p_{x,y})(v_y^+)^2 ,
    \\
    \sim & \frac{\left(\frac1n\sum_{x \in [n]} (\sqrt{\tfrac{n}{\w}}w_x^-)^2(\sqrt{\tfrac{n}{\w}}w_x^+) \right)\!\!\left(\frac1n\sum_{y \in [n]}(\sqrt{\tfrac{n}{\w}}w_y^-)(\sqrt{\tfrac{n}{\w}}w_y^+)^2\right)}{\left(\frac1n\sum_{x \in [n]} (\sqrt{\frac{n}{\w}} w_x^+)(\sqrt{\frac{n}{\w}} w_x^-)\right)^2}.
\end{split}
\end{equation}
Using the hypothesis, the sums converge to integrals and we get \eqref{eq:sigma-integral}.
\end{proof}

\section{Analysis of outliers, higher-rank case}
 \label{sec:sp-outliers-r}
 
\subsection{Existence of outliers} 
\label{subsec:sp-existence-outliers-r}

For the model defined in \eqref{eq:rank-m}, Proposition \ref{prop:prop-W} and Lemma \ref{lem:lem-W} still hold, with the corresponding definition of $\centered$. To establish Theorem \ref{thm:LLN-A-rank-m}, we need to adapt the Bauer--Fike step, by employing directly Theorem \ref{thm:Bauer-Fike}, with the choice $\deterministic=\E[A_n]$. We need to show $$\eps_n=\|P_n\|\|P_n^{-1}\|\|\centered\| = \Ohp(\sqrt{\scaling}),$$ where $P_n$ is a diagonalizing change of basis for $\E[A_n]$.
In this setting, the change of basis $P$ can be chosen to be
 $$P_n= (v_1^+,\dots,v_r^+, e_{r+1}, \dots, e_{n}),$$
 where $(e_{l})_{r+1 \le l \le n}$ is an orthonormal basis of ${\rm Span}(v_1^-,\dots,v_r^-)^{\perp}$. $P_n$ is not an orthogonal matrix, but it holds $\|P_n\|_{F}=\sqrt{n}$. Moreover, considering the matrix $X_n=(v_1^-,\dots,v_r^-, e_{r+1}, \dots, e_{n})$,
 we have that $X_n^*P_n$ is lower triangular with unit determinant, so that $|\det(P_n)|= |\det(X_n)^{-1}|$.
 It holds for any $i \le r$
 % \begin{equation} \label{eq:used}
 % \max_{i,j \le r} {\rm dist}(v_{i}^-, v_{j}^-)^2
 % % =\max_{i,j \le m} \sqrt{\sum_{x \in [n]}\left(v_{i}^-(x)- v_{j}^-(x)\right)^2}
 % =\frac1n\cdot \max_{i,j \le r} \sum_{x \in [n]}\left(\sqrt{n}\,v_{i}^-(x)- \sqrt{n}\,v_{j}^-(x)\right)^2,\end{equation}
  \begin{equation} \label{eq:used}
\|v_i^-\|^2
 =\frac1n\cdot\sum_{x \in [n]}(\sqrt{n}v_{i}^-(x))^2=O(1),\end{equation}
 thanks to Assumption \ref{ass:rank-m}. Then
 $$|\det(X_n)|= \prod_{l=0}^{r-1} {\rm dist}(v_{l+1}^-, {\rm Span}(v_1^-,\dots,v_l^-))\le \prod_{l=0}^{r-1}\|v_{l}^-\|= O(1).$$
We have the bound (which is the main contribution in \cite{GEJ})
\begin{equation} \label{eq:X_n}
    \|P_n\|\|P_n^{-1}\|\le \frac{2}{|\det(P_n)|}\left(\frac{\|P_n\|_F}{\sqrt n}\right)^n = {2|\det(X_n)|}.
\end{equation}
Then, by Proposition \ref{prop:prop-W} we conclude $\eps_n=\Ohp(\sqrt{\scaling})$.

\subsection{Fluctuations around the mean}
Fix $l \in \{1,\dots,r\}$. To simplify the notation, let $\lambda_l=\lambda_{l}(A_n)$. We also consider the $r \times r$ matrix with entries
\begin{equation} \label{eq:V-diag}
     V_n(i,j)\coloneq\scaling \sqrt{\theta_i \theta_j} (v_i^-)^t\left(I_n - \frac{\centered}{\lambda_l}\right)^{-1}v_j^+ \,\one_{\{\|C_n\|<\lambda_l\}}\, \qquad i,j =1,\dots,r.
\end{equation}
Notice that, thanks to Theorem \ref{thm:LLN-A-rank-m} and by the conditions on $(v_i^\pm)_{i \le r}$, it holds that
\begin{equation}\label{eq:V-id}
\left\|V_n-s_n\operatorname{Diag}(\theta_1,\ldots,
\theta_r)\right\|
=O\left(s_n\frac{\|C_n\|}{|\lambda_l|}\right)
=\Ohp(\sqrt{s_n}).
\end{equation}
that is, $V_n$ is an additive perturbation of a diagonal matrix whose distinct eigenvalue gaps have order $s_n$,
 % and it is diagonalizable (say, because it is with high probability strictly dominant and hence invertible with distinct eigenvalues).
since $\theta_1>\cdots>\theta_r$ and $r$ is fixed.
% the deterministic gaps between the diagonal entries of $s_n\operatorname{Diag}(\theta_1,\ldots,\theta_r)$ are of order $s_n$. 
% By \eqref{eq:V-id}, the perturbation has controlled norm and
Then, with very high probability, the corresponding Gershgorin discs are pairwise disjoint.  Each disc contains exactly one eigenvalue of $V_n$.  In particular, these eigenvalues are distinct and real, and $V_n$ is diagonalizable.
More precisely, the outliers of $A_n$ are recovered as the eigenvalues of $V_n$, as the following lemma states.

\begin{lemma} \label{lem:A-V}
With very high probability it holds $\lambda_l(A_n)=\lambda_l(V_n)$.
\end{lemma}

\begin{proof}%[Proof of Lemma \ref{lem:A-V}]
    Let $v$ be a right unit eigenvector of $A_n$ corresponding to the eigenvalue $\lambda_l=\lambda_l(A_n)$, 
    \begin{equation} \label{eq:lambda-v}
        \lambda_l v = \centered v + \scaling\sum_{j=1}^r \theta_j v_j^+ (v_j^-)^tv.
    \end{equation}
    Reasoning as in \eqref{eq:implications}, \wehp it holds 
    \begin{equation}
    \begin{split}
         \lambda_l v
        = \left(I_n-\frac{\centered}{\lambda_l}\right)^{-1} \scaling\sum_{j=1}^r \theta_jv_j^+ (v_j^-)^t v,
    \end{split}
    \end{equation}
    so that, pre-multiplying by $\sqrt{\theta_i} (v_i^-)^t$, for $i=1,\dots,r$, and recalling the definition \eqref{eq:V-diag}, we have
    \begin{equation}
        \lambda_l \,\sqrt{\theta_i} (v_i^-)^t v   = \sum_{j=1}^r V_{n}(i,j)  \sqrt{\theta_j}(v_j^-)^t v,\qquad i=1,\dots,r.
    \end{equation}
 Set $u^\pm=(\sqrt{\theta_1} (v_1^\pm)^t v,\dots, \sqrt{\theta_r} (v_r^\pm)^t v)^t$. Then $u^-$ is a candidate eigenvector of $V_n$ with eigenvalue $\lambda_l$. We have to show that it is not the null vector. Pre-multiplying \eqref{eq:lambda-v} by $v^t$,
    \begin{equation}
        \lambda_l = v^t \centered v +\scaling u^+  u^-.
    \end{equation}    
    Since $\lambda_l$ is of order $\scaling$ w.v.h.p.\ and $v^t \centered v$ has lower order (thanks to Proposition \ref{prop:prop-W} it holds $\|\centered\|=\Ohp(\sqrt{\scaling})$), we deduce that $u^-$ is non-zero. This shows that $\lambda_l(A_n)$ is in the spectrum of $V_n$. To prove the claim, it remains to use Gershgorin's  Theorem (\cite[Theorem 1.6]{Varga2004} or \cite[Fact 5.1]{CCH}) as in \cite[Lemma 5.2]{CCH}, after having noticed that $A_n$ does not need to be symmetric.
\end{proof}
Since we have established that the eigenvalues of $A_n$ are \wehp real, with an abuse of notation we will consider the real sequence $(\lambda_l(A_n))_{1\le l \le r}$ to be ordered in monotone decreasing way.
To proceed with the analysis,
the decomposition \eqref{eq:lambda-decomp} needs to be adapted to the $r$-dimensional context. This will be the aim of the next Lemmata.
We write $$V_n = \sum_{k=0}^{+\infty}\frac{V_{k,n}}{\lambda_l^k},$$
where for every $k \in \N$, $V_{k,n}$ is the matrix with entries 
\begin{equation} 
V_{k,n}(i,j)\coloneqq \scaling \sqrt{\theta_i\theta_j}(v_i^-)^t\left( {\centered}\right)^{k}v_j^+ \qquad i,j =1,\dots, r.
\end{equation}
Let now $L=\lceil\log(n)\rceil$  and consider the following fixed point equation
\begin{equation}
    x=h_l(x)\coloneqq \lambda_l\left(\sum_{k=0}^L \frac{\E[V_{k,n}]}{x^k}\right),
\end{equation}
which generalizes the one in \eqref{eq:fixed-point}.
Letting $x= y \scaling$ for $y \in (0,+\infty)$, by Lemma \ref{lem:1} it holds 
\begin{equation}
    \left\|\sum_{k=2}^L \frac{\E[V_{k,n}]}{(y \scaling)^k }\right\|\le \sum_{k=2}^{+\infty} (y \scaling)^{-k} (K_1 \scaling)^{k/2+1}=\left(\frac{K_1}{y}\right)^2(1+O(\scaling^{-1/2})).
\end{equation}
As a consequence, by definition of $V_{0,n}$ and the properties of $(v_i^\pm)_{i \le r}$, $$\scaling^{-1} \sum_{k=0}^{L} (y\scaling)^{-k} \E[V_{k,n}]= {\rm Diag}(\theta_1,\dots,\theta_r)+O(1/\scaling).$$
Since the eigenvalues are real and distinct, with spacing of order $\scaling$, for every choice of $x=y\scaling$, the function $(0,+\infty) \ni y \mapsto \sum_{k=0}^L \E[V_{k,n}](y\scaling)^{-k}$ is continuous and this means that at the scale $x=y\scaling$, $h_l(x)$ is itself continuous, and in particular $h_l(y \scaling)=\theta_l\scaling(1+o(1))$. Then we have that, for $y< \theta_l$ and large $n$, $y \scaling < h_l(y\scaling)$, and the converse for $y >\theta_l$. Thus, the equation must have a solution $\tilde \lambda_l$ at the scale $\scaling$.

\begin{lemma} \label{lem:tilde-m}
       In the present setting, it holds
    % \begin{equation}
        $\lambda_l-\tilde\lambda_l=\Ohp\left(\frac{\left\|V_{1,n}\right\|}{\scaling}+\sqrt{\frac{\scaling}{n}}\right).$
    % \end{equation}
\end{lemma}
\begin{proof}
   Let $\deterministic^{(0)}=V_n$.
    By \eqref{eq:V-id},
$\|S_n^{(0)}-s_n {\rm Diag}(\theta_1,\dots,\theta_r)\|=O_{v.h.\mathbb P}(\sqrt{s_n})$, that is $S_n^{(0)}$ is 
    an additive perturbation of a diagonal matrix with well spaced real eigenvalues. 
If $q_l=e_l+w_l$ denotes the eigenvector for the $l$-th eigenvalue $\mu_l$, normalized so that $q_l^t e_l =1$ (in particular $(w_l)_l=0$), and $E=S_n^{(0)}-\scaling\operatorname{Diag}(\theta_1,\dots,\theta_r)$, we have
\begin{equation}
    0=(S^{(0)}_n - \mu_l)q_l
    =
    (\scaling{\rm Diag}(\theta_1,\dots,\theta_r)+E-\mu_l)(w_l+e_l) 
\end{equation}
which implies that on the other components $ j\neq l$, being $(e_l)_j=0$
\begin{equation}
     (\scaling\theta_j- \mu_l)(w_l)_{j}=
     -(Eq)_j.
\end{equation}
    In particular, we get 
    \begin{equation}
        \max_{j \neq l} |(w_l)_j| \le \max_{j \neq l}\frac{1}{|\scaling\theta_j-\mu_l|} \|E\| \| e_l+w_l\|
    \end{equation}
    and \begin{equation}
         \|w_l\| \le \sqrt{r} \max_{j \neq l}\frac{1}{|\scaling\theta_j-\mu_l|} \|E\| (1+\|w_l\|),
    \end{equation}
    from which we easily derive that $\|w_l\|= \Ohp({\scaling}^{-1/2})$.
    Then the eigenvectors of $S_n^{(0)}$ turn to be approximated, up to an additive error $\Ohp  \left({\scaling}^{-1/2}\right)$, by the eigenvectors of the canonical basis. 
    As a consequence eigenvectors of $\deterministic^{(0)}$ are approximately orthogonal.
    Consider the following $r \times r$ matrices:
     \begin{alignat}{2}
% \begin{equation}
% \begin{split}
    \deterministic^{(1)}=& \ \sum_{k=0}^L\frac{V_{k,n}}{\lambda_l^k}, \qquad  &&    
    %\text{(high exponent)}
    \\
    \deterministic^{(2)}=& \ V_{0,n}+\frac{V_{1,n}}{\lambda_l}+\sum_{k=2}^L\frac{\E[V_{k,n}]}{\lambda_l^k}, \qquad &&
    %\text{(centering)}
    \\
    \deterministic^{(3)}=& \ \sum_{k=0}^L\frac{\E[V_{k,n}]}{\tilde \lambda_l^k}     
    . \qquad && %\text{(main contribution)} 
% \end{split}
% \end{equation}
\end{alignat}
It is not difficult to see that for $\ell=1,2,3$, the same diagonal approximation holds and $\deterministic^{(\ell)}$ is a random perturbation of the matrix $\deterministic^{(\ell-1)}$.
Then it is possible to apply sequentially Theorem \ref{thm:Bauer-Fike} with the choices $H^{(\ell)}_n=\deterministic^{(\ell)}-\deterministic^{(\ell-1)}$ and get that with high probability 
\begin{equation}
    |\lambda_l(\deterministic^{(\ell)})-\lambda_l(\deterministic^{(\ell-1)})|\le \|P^{(\ell)}_n\|\|(P^{(\ell)}_n)^{-1}\|\|H^{(\ell)}_n\|,
\end{equation}
where $P^{(\ell)}_n$ has as columns the (unit) eigenvectors of $\deterministic^{(\ell-1)}$. 
Being the matrix $P^{(\ell)}_n$ well conditioned, since its column are approximately orthogonal, being small perturbations of the canonical basis
% \textcolor{blue}{to erase: Because of the bound
%  $$\|P^{(\ell)}_n\|\|(P^{(\ell)}_n)^{-1}\|\le \frac{2}{|\det(P^{(\ell)}_n)|}\left(\frac{\|P^{(\ell)}_n\|_F}{\sqrt r}\right)^r= \lesssim 2,$$ (which comes from \cite{GEJ})}
we get 
\begin{equation}
    |\lambda_l-\tilde \lambda_l|=|\lambda_l(\deterministic^{(0)})-\lambda_l(\deterministic^{(3)})|\le 2(\|H^{(1)}_n\|+\|H^{(2)}_n\|+\|H^{(3)}_n\|).
\end{equation}
Hence, it is enough to bound the r.h.s.. 
$\|H^{(1)}_n\|$ is bounded as in \eqref{eq:R1}, in the same way as $R^{(1)}$ was bounded in the rank-one case.
To bound $\|H^{(2)}_n\|$, it is sufficient to observe that
\begin{equation}
    \|H^{(2)}_n\| \le \sum_{k=2}^{L} \left\|\frac{V_{k,n}-\E[V_{k,n}]}{\lambda_l^k}\right\|
    \le  K_5 \max_{i,j \le r} \sum_{k=2}^{L} \left|\frac{V_{k,n}(i,j)-\E[V_{k,n}(i,j)]}{\lambda_l^k}\right|,
\end{equation}
and then employ Lemma \ref{lem:2} as in \eqref{eq:R2} to bound the terms in the \rhs of the above display,
uniformly in $i,j$.
Finally, to bound $\|H^{(3)}_n\|$,
%we employ Lemma \ref{lem:H^{(3)}_n}.
notice that
    \begin{equation}  \label{eq:thesis-H3}
        \|H^{(3)}_n\|=\left\| \frac{V_{1,n}}{\lambda_l}+\sum_{k=2}^L\E[V_{k,n}]\left(\frac{1}{\lambda_l^k}-\frac{1}{\tilde \lambda_l^k}\right)\right\|.
    \end{equation}
Reasoning as in the proof of Lemma \ref{lem:fixed-point}, we can bound the \rhs of \eqref{eq:thesis-H3} by
\begin{equation} \label{eq:lambda-tilde-m}
     \left\| \lambda_l^{-1} V_{1,n}\right\|+
     |\lambda_l-\tilde \lambda_l|\sum_{k=2}^L \left\|\E[ V_{k,n}]\right\| \frac{\sum_{j=0}^{k-1}{\tilde\lambda}_l^{j}\lambda^{k-j-1}}{\lambda^{k}{\tilde\lambda}_l^{k}}\,,
\end{equation}
which in the fashion of \eqref{eq:frac2} and thanks to Theorem \ref{thm:LLN-A-rank-m}, needed to estimate the first term, implies
\begin{equation} \label{eq:H3}
\|H^{(3)}_n\| = \Ohp\left(\frac{\left\|V_{1,n}\right\|}{\scaling}\right)+|\lambda_l-\tilde{\lambda}_l| \Ohp({\scaling}^{-1}).
\end{equation}
Putting all estimates together, we conclude the proof of Lemma \ref{lem:tilde-m}, since we get
\begin{equation}
    |\lambda_l-\tilde\lambda_l| (1+\Ohp(\scaling^{-1})) = \Ohp\left(\frac{\left\|V_{1,n}\right\|}{\scaling}+\sqrt{\frac{\scaling}{n}}\right). \qedhere
\end{equation}

\end{proof}
Finally, we can refine the previous result to the following one, which is analogous to Lemma \ref{lem:fixed-point}.
We recall that in this section the $r$ real outliers are ordered in decreasing order.
\begin{lemma} \label{lem:63}
   In the present setting, it holds
    \begin{equation}
        \lambda_l - \tilde \lambda_l 
        =\scaling \theta_l \frac{(v_l^-)^t \centered v_l^+}{ \lambda_l} 
        + \ohp \left(\frac{\left\|V_{1,n}\right\|}{\scaling}+\sqrt{\frac{\scaling}{n}}\right).
    \end{equation}
\end{lemma}
\begin{proof}
We apply the Bauer--Fike argument again, but changing now $\deterministic^{(3)}$ with  $$\tildedeterministic^{(3)}=\frac{V_{1,n}}{\lambda_l} +\sum_{k=0}^L\frac{\E[V_{k,n}]}{\tilde \lambda_l^k} =\frac{V_{1,n}}{\lambda_l}+\deterministic^{(3)}.$$
We get that
\begin{equation}
\begin{split}
    \left|\lambda_l-\lambda_l\left(\frac{V_{1,n}}{\lambda_l} +\sum_{k=0}^L\frac{\E[V_{k,n}]}{\tilde \lambda_l^k} \right)\right|
    &\le
    2(\|H^{(1)}_n\|+\|H^{(2)}_n\|+\|\tilde H^{(3)}_n\|).
    \end{split}
\end{equation}
and now, letting $\tilde H^{(3)}_n= \tildedeterministic^{(3)}-\deterministic^{(2)}$, it turns that 
\begin{equation}
    \|\tilde H^{(3)}_n\| 
    = |\lambda_l-\tilde\lambda_l|\Ohp(\scaling^{-1})
    = \Ohp \left(\frac{\left\|V_{1,n}\right\|}{\scaling}+\sqrt{\frac{\scaling}{n}}\right)\Ohp(\scaling^{-1})
    =\ohp \left(\frac{\left\|V_{1,n}\right\|}{\scaling}+\sqrt{\frac{\scaling}{n}}\right),
\end{equation}
where the first asymptotic estimate comes from \eqref{eq:lambda-tilde-m} and \eqref{eq:H3} without the first error term (due to the definition of $\deterministic^{(3)}$), and the second one follows from Lemma \ref{lem:tilde-m}.
Let us now consider the matrices
\begin{align*}
    \tilde{H}\coloneqq \tildedeterministic^{(3)} - \tildedeterministic^{(3)}(l,l)I_r ,\qquad \qquad
    \tilde{M}\coloneqq \tildedeterministic^{(3)}-\frac{V_{1,n}}{\lambda_l} - \left(\tildedeterministic^{(3)}(l,l)-\frac{V_{1,n}(l,l)}{\lambda_l}\right)I_r,%=\deterministic^{(3)}-\deterministic^{(3)}(l,l)I_r,
\end{align*}
obtained by adding and subtracting to $\tildedeterministic^{(3)}$ and $\tildedeterministic^{(3)}-\frac{V_{1,n}}{\lambda_l}$ multiples of the identity (we highlight that $V_{1,n}(l,l)=\scaling \theta_l (v_l^-)^t \centered v_l^+$ is the $(l,l)$ entry of $V_{1,n}$). Since this only translates eigenvalues, it follows that
\begin{equation}
    \lambda_l\big(\tildedeterministic^{(3)} \big)
    = \lambda_l(\tilde{H})+\tildedeterministic^{(3)}(l,l)
    =\lambda_l(\tilde{H})+\frac{V_{1,n}(l,l)}{\lambda_l}+\lambda_l\left(\tildedeterministic^{(3)}-\frac{V_{1,n}}{\lambda_l}\right)-\lambda_l(\tilde{M}),
\end{equation}
which means, recalling that $\tildedeterministic^{(3)}-\frac{V_{1,n}}{\lambda_l}=\deterministic^{(3)}$ and $\lambda_l\big(\deterministic^{(3)}\big)=\tilde\lambda_l$,
\begin{equation}
    \lambda_l\left(\frac{V_{1,n}}{\lambda_l} +\sum_{k=0}^L\frac{\E[V_{k,n}]}{\tilde \lambda_l^k}\right)=\tilde\lambda_l+\frac{V_{1,n}(l,l)}{\lambda_l}+\lambda_l(\tilde{H})-\lambda_l(\tilde{M}).
\end{equation}
To conclude the proof of the lemma, we need to show
$$|\lambda_l(\tilde{H})-\lambda_l(\tilde{M})|= \ohp\left(\frac{\|V_{1,n}\|}{\scaling}\right).$$
This follows by the same argument as in the proof of \cite[Lemma 5.8, after Eq.\ (5.26)]{CCH} and Lemma \ref{lem:k=1}, adapted to the rank-$r$ setting. The argument applies here using the previously established Bauer–Fike bounds for the isolated eigenvalues and controlling the off-diagonal rows and columns separately.
\end{proof}
\begin{proof}[Proof of Theorem \ref{thm:CLT-A-rank-m}]
It holds 
\begin{equation}
    \frac{\E\|V_{1,n}\|}{\scaling}=O\left(\frac{ \sum_{i,j =1}^r 
    % \sqrt{\theta_i\theta_j}
    \scaling\E|(v_i^-)^t\centered v_j^+|}{\scaling}\right)
    =
    O\left(\sum_{i,j =1}^r 
    % \sqrt{\theta_i\theta_j}
    \sqrt{\var((v_i^-)^t {\centered}v_j^+)}\right)
    ,
\end{equation}
so that, reasoning as in Lemma \ref{lem:5}, we have $\|V_{1,n}\|/\scaling=\Op ( \sqrt{\scaling/n})$. Then,
Lemma \ref{lem:63} and a higher-rank analogue of Lemma \ref{lem:R} (to substitute $\tilde \lambda_l$ with the expectation of $\lambda_l$) imply that
 \begin{equation}
 \begin{split}
      \sqrt{\frac{n}{\scaling}} \left( \lambda_l - \E[ \lambda_l ]\right)
        &=\sqrt{\frac{{n}}{{\scaling}}} \frac{\scaling \theta_l }{ \lambda_l}   (v_l^-)^t \centered v_l^+ \,\,
        +\op \left(1\right)\\
        &=\sqrt{\frac{{n}}{{\scaling}}} \frac{\scaling \theta_l }{ \tilde \lambda_l}   (v_l^-)^t \centered v_l^+ \,\,
        +\op \left(1\right)\\
        &= \sqrt{\frac{{n}}{{\scaling}}} (v_l^-)^t  \centered v_l^+ \,\,
        + \op \left(1\right),
    \end{split}
    \end{equation}
    where in the second line we have used the natural extension of \eqref{eq:49}, that is 
    \begin{equation}
     \left|\scaling\theta_l\left(\frac{1}{\phantom{\tilde\lambda} \!\!\!\!\!\ \lambda_l}-\frac{1}{\tilde\lambda_l}\right)(v^-_i)^t \centered v^+_j\right| = \ohp(|\lambda_l-\tilde\lambda_l|)=\op\left(\sqrt{\frac{\scaling}{n}}\right).
    \end{equation}
    and later we used $\tilde\lambda_l \sim \theta_l \scaling$. Since $\var(C_{xy})=p_{x,y}(1-p_{x,y})\sim p_{x,y}$, a straightforward calculation shows that, for $i,j \le r$,
    \begin{equation}
    \begin{split}
        {\frac{{n}}{{\scaling}}}& \, \cov\Big((v_i^-)^t \ \centered  v_i^+,(v_j^-)^t \centered  v_j^+\Big)= n\!\!\!\sum_{x,y\in[n]}\!\!\! v_i^-(x)v_j^-(x) \bigg(\!\sum_{k=1}^r \theta_k v_k^+(x) v_k^-(y)\!\!\bigg)v_i^+(y)  v_j^+(y)+o(1)\\
        =&\frac{1}{n^2}\!\!\!\sum_{x,y\in[n]} \!\!(\sqrt{n}v_i^-(x))(\sqrt{n}v_j^-(x)) \bigg(\sum_{k=1}^r \theta_k (\sqrt{n}v_k^+(x)) (\sqrt{n}v_k^-(y))\bigg)(\sqrt{n}v_i^+(y)) ( \sqrt{n}v_j^+(y))+o(1),
    \end{split}
    \end{equation}
    which under Assumption \ref{ass:rank-m}, asymptotically provides the expression in \eqref{eq:cov}.
    % , and thus completes the proof
 \end{proof}

\printbibliography

\end{document}